\documentclass[11pt]{amsart}

\usepackage{amsmath, amsfonts, amssymb, amsthm, mathtools}
\usepackage{tikz-cd}
\usepackage[utf8]{inputenc}
\usepackage{csquotes}
\usepackage{subcaption}
\usepackage[shortlabels]{enumitem}
\usepackage[hidelinks]{hyperref}
\usepackage[english]{babel}
\usepackage[giveninits, citestyle=numeric-comp,sorting=nty,sortcites, backend=biber]{biblatex}

\bibliography{main.bib}
\AtEveryBibitem{
	\clearfield{doi}
	\clearfield{url}
	\clearfield{month}
	\clearfield{day}
	\clearfield{issn}
	\clearfield{pages}
	\clearfield{publisher}
	\clearfield{isbn}
	\clearfield{pagetotal}
	\clearfield{note}
}
\DeclareFieldFormat*{urldate}{}
\DeclareFieldFormat*{title}{#1}
\renewbibmacro{in:}{}
\DeclareFieldFormat*{volume}{vol. #1}
\DeclareFieldFormat*{number}{no. #1}
\renewbibmacro*{volume+number+eid}{%
	\printfield{volume}%
	\setunit*{\addcomma\space}
	\printfield{number}%
	\setunit{\addcomma\space}%
	\printfield{eid}
}
\DeclareFieldFormat[misc]{date}{(#1)}

\usepackage{todonotes}

\newtheorem{theorem}{Theorem}[section]
\newtheorem{proposition}[theorem]{Proposition}
\newtheorem{corollary}[theorem]{Corollary}
\newtheorem{lemma}[theorem]{Lemma}
\newtheorem{definition}[theorem]{Definition}
\newtheorem{remark}[theorem]{Remark}
\newtheorem{example}[theorem]{Example}

\newtheorem{claim}[theorem]{Claim}

 \DeclareMathOperator{\mcg}{Map}
 \DeclareMathOperator{\pmcg}{PMap}
 \DeclareMathOperator{\homeo}{Homeo}

\DeclareMathOperator{\diffeo}{Diff}

\newcommand{\integers}{\mathbb{Z}}

 \DeclareMathOperator{\cut}{Cut}

\newcommand{\forget}{\mathcal{F}}
\newcommand{\centralizer}{\mathcal{Z}}

\newcommand{\cerrado}[1]{\overline{U_{#1}}}
\newcommand{\cota}{ 6\cdot 2^{g-4}}

\usepackage{xcolor}

\newcommand{\new}[1]{#1}
\newcommand{\neww}[1]{#1}

\title{Homomorphisms between pure mapping class groups}
\author{Rodrigo De Pool}
\thanks{The author acknowledges financial support from the grant  CEX2019-000904-S funded by  MCIN/AEI/ 10.13039/501100011033 and from the grant PGC2018-101179-B-I00.}

\begin{document}

\begin{abstract}
	Let $S$ and $S'$ be orientable finite-type surfaces of genus $g\geq 4$ and $g'$, respectively. We prove that every multitwist-preserving map between pure mapping class groups $\pmcg(S)\to \pmcg(S')$  is induced by a  multi-embedding. As an application, we classify all homomorphisms $\pmcg(S)\to \pmcg(S')$ for $g\geq 4$ and $g' \leq  \cota$.
\end{abstract}

\maketitle


\section{Introduction}\label{introd}


Let $S$ be an orientable finite-type surface possibly with punctures or boundary.  The \emph{mapping class group} $\mcg(S)$ is the group of orientation preserving homeomorphisms of $S$ up to homotopy. The \emph{pure mapping class group} $\pmcg(S)$ is the subgroup of the mapping class group acting trivially on the set of punctures of $S$.

A central objective in the study of mapping class groups is to connect the geometry of the surface $S$ to the algebraic structure of the group $\pmcg(S)$. In this article, we focus on homomorphisms $\pmcg(S)\to \pmcg(S')$ between pure mapping class groups. Our first result relates these homomorphisms to maps between the diffeomorphism groups. 

Let $\diffeo_c^+(S)$ be the group of orientation preserving diffeomorphisms of $S$ compactly supported on $S\setminus \partial S$, and let $\pi:\diffeo^+_c(\cdot) \to \pmcg(\cdot)$ be the map projecting each diffeomorphism to its homotopy class. It is a question of Aramayona and Souto \cite[Question 5]{aramayona_rigidity_2016} whether every homomorphism $\varphi:\pmcg(S)\to \pmcg(S')$ admits a lift $\overline{\varphi}$  to the diffeomorphism groups, that is, whether we have a commutative diagram
\begin{center}
	\begin{tikzcd}
		\diffeo^+_c(S) \arrow[r, "\overline{\varphi}", dashed] \arrow[d, "\pi"] & \diffeo^+_c(S') \arrow[d, "\pi"] \\
		\pmcg(S) \arrow[r, "\varphi"]                          & \pmcg(S').    
	\end{tikzcd}
\end{center}
The next theorem yields a positive answer to  \cite[Question 5]{aramayona_rigidity_2016} under suitable conditions on the genus of the surfaces. 

\begin{theorem}\label{thm:version_souto}
	Let $S, S'$ be two orientable finite-type surfaces of genus $g\geq 4$ and $g' \leq 6\cdot 2^{g-4}$. Any homomorphism $\varphi:\pmcg(S)\to \pmcg(S')$ admits a lift $\overline{\varphi}: \diffeo^+_c(S)\to \diffeo^+_c(S')$ to the diffeomorphism groups.
\end{theorem}

We emphasize that Theorem \ref{thm:version_souto} only imposes conditions on the genus of the surfaces, and has no assumptions on the number of boundary components or punctures.

A theorem of Hurtado \cite{hurtado_continuity_2015} states that every homomorphism from the identity component of $\diffeo^+_c(S)$ to the identity component of  $\diffeo^+_c(S')$ is an `extended topologically diagonal map'.  Although we do not rely on Hurtado's result, the lifts $\overline{\varphi}: \diffeo^+_c(S)\to \diffeo^+_c(S')$ from Theorem \ref{thm:version_souto} admit an analogous description. Namely, we say that a homomorphism $\overline{\varphi}: \diffeo^+_c(S)\to \diffeo^+_c(S')$  is \emph{induced by a multi-embedding} if there is a set of disjoint  embeddings $\{\iota_i: S\hookrightarrow S'|\; i=1,2,\dots,n\}$  such that $\overline{\varphi}$ is given by $\phi\mapsto \hat{\phi}$, where $\hat{\phi}$ is defined as
\[ 
\hat{\phi}(x) = 
\begin{cases}
	x & x\not \in \iota_i(S) \; \forall i,\\ 
	\iota_i \circ \phi \circ \iota_i^{-1}(x) & x\in \iota_i(S).
\end{cases}
\]
If $\overline{\varphi}$ is a lift of the homomorphism $\varphi:\pmcg(S)\to \pmcg(S')$, we also say that $\varphi$ is  induced by a multi-embedding.

To prove Theorem \ref{thm:version_souto} we actually show that, within the same genus bounds, every  homomorphism $\pmcg(S)\to\pmcg(S')$ is induced by a \emph{multi-embedding}.

\begin{theorem}\label{thm:clasificacion}
	Let $S,S'$ be two orientable finite-type surfaces of genus  $g\geq 4$ and $g' \leq \cota$. If $\varphi:\pmcg(S) \to \pmcg(S')$ is a nontrivial homomorphism, then  $\varphi$ is induced by a multi-embedding.
\end{theorem}

Theorem \ref{thm:clasificacion} extends some already known results. For instance, when $g'=g$, we recover results of Ivanov \cite{ivanov_automorphisms_1988} and Ivanov-McCarthy \cite{ivanov_automorphism_1997}  on injective endomorphisms of the  mapping class group. Also, if  $g'\leq2g-1$, the statement recovers Aramayona-Souto's classification \cite{aramayona_homomorphisms_2013}, which asserts that every nontrivial homomorphism $\pmcg(S) \to \pmcg(S')$ is induced by an embedding. Another immediate consequence of Theorem \ref{thm:clasificacion} is the following generalization  of \cite[Theorem 1]{chen_constraining_2021} by Chen-Lanier.
\begin{corollary}
	Let $S_{g}, S_{g'}$ be two closed surfaces of genus $g\geq 4$ and $g' \leq \cota$. If $g\ne g'$, every homomorphism $\mcg(S_g) \to \mcg(S_{g'})$  is trivial. 
\end{corollary}

We stress the importance of the genus bounds in Theorem \ref{thm:clasificacion}. For instance, for some $g'$ large enough, we may consider a finite quotient of $\pmcg(S)$ that embeds in $\pmcg(S')$. Such a map cannot be induced by a multi-embedding. Similarly, if $S$ has genus $g\leq 2$ and $S'$ is any closed surface, there are homomorphisms $\pmcg(S)\to \pmcg(S')$ with finite image $\mathbb{Z}/2\mathbb{Z}$. Again, such maps are not induced by multi-embeddings.  

More interestingly, the statement of Theorem \ref{thm:clasificacion} requires some exponential upper bound on  $g'$. Indeed, if $g' \geq (g-1)\cdot 2^{2g}+1$, there are injections between mapping class groups  that are not induced by multi-embeddings. In Appendix \ref{appendix:prueba_lema} we produce such homomorphisms following a construction of Ivanov-McCarthy \cite{ivanov_injective_1999}.

The proof of Theorem \ref{thm:clasificacion} relies on the following result of Bridson  \cite{bridson_semisimple_2010}: If $S$ has genus $g\geq 3$, then every homomorphism $\varphi: \pmcg(S)\to \pmcg(S')$ sends Dehn twists to roots of multitwists. Building on this, we show that the genus bounds imply that  $\varphi$ is  \emph{multitwists-preserving}, meaning that the image of every non-separating Dehn twist is a multitwist. Then,  that $\varphi$ is induced by a multi-embedding follows from the next result. 

\begin{theorem}\label{thm:preserva_multis_emb}
	Let $S,S'$ be two orientable finite-type surfaces and let $g\geq 3$ be the genus of $S$. If   $\varphi:\pmcg(S)\to \pmcg(S')$ is a multitwist-preserving map, then  $\varphi$  is induced by a multi-embedding.
\end{theorem}

A key part in the proof of Theorem \ref{thm:preserva_multis_emb} involves  understanding when multitwists satisfy the braid relation, for which we rely on our earlier paper \cite{de_pool_braided_2024}. Now,  to understand \emph{non} multitwist-preserving homomorphisms $\pmcg(S)\to \pmcg(S')$ one probably requires a convenient description of roots of multitwists and their relations. These elements are quite mysterious with nontrivial examples discovered not so long ago  \cite{margalit_dehn_2009}. Thus, a complete characterization of homomorphisms  $\pmcg(S)\to \pmcg(S')$ remains elusive. 

To conclude, we note that the results presented here will be used in a forthcoming paper \cite{holomorphic_moduli} to classify holomorphic maps between moduli spaces, within the same range as  Theorem \ref{thm:clasificacion}.

\subsection*{Plan of the paper} 
In Section \ref{sec:preliminares} we introduce some definitions and results used throughout the article. In Section \ref{sec:preserva_giros} we prove a weaker version of Theorem \ref{thm:preserva_multis_emb}.  In Section \ref{sec:preserva_multigiros} we prove  Theorem \ref{thm:preserva_multis_emb}.  In Section \ref{sec:clasificacion} we prove  Theorem \ref{thm:clasificacion}. In Section \ref{sec:lift} we prove Theorem \ref{thm:version_souto}. In Appendix  \ref{appendix:prueba_lema}  we construct homomorphisms not induced by a multi-embedding and compute the genus of the associated surfaces.

\subsection*{Acknowledgments}
I would like to thank my supervisor, Javier Aramayona, for suggesting the problem and, specially, for the patience and guidance. I would also like to thank Juan Souto for reading through the first draft and providing useful suggestions. 

\section{Preliminaries}\label{sec:preliminares}
In this section we introduce some basic definitions and results about mapping class groups that are used throughout the article. For a thorough introduction to mapping class groups we refer the reader to the book of Farb-Margalit \cite{farb_primer_2012}. 

By a \emph{surface} we mean a connected and orientable  2-manifold with (possibly empty) compact boundary.  All  surfaces in this article are finite-type surfaces, that is, surfaces with finitely generated fundamental group. Alternatively, we may define a finite-type surface $S$ as a compact surface minus a finite set of points in its interior.  The points removed are called \emph{punctures} of $S$. The genus $g$, the number of boundary components $b$ and the number of punctures $p$, characterize the surface $S$ up to homeomorphism. 

By a \emph{curve} on $S$ we mean the homotopy class of a simple closed curve that does not bound a disk or a puncture. A curve on $S$ is \emph{non-separating} if  there exists a representative whose complement in $S$ is connected. Given two curves $a$ and $b$ the \emph{intersection number} $i(a,b)$ is the minimum number of intersection points between a representative of $a$ and a representative of $b$. Two curves are said to be \emph{disjoint} if $i(a,b)=0$.  A \emph{multicurve} on $S$ is a set of pairwise disjoint curves. Recall that if $S$ has genus $g$, $b$ boundary components and $p$ punctures, then every multicurve $\eta$ on $S$ contains at most  $3g-3+p+b$ curves.

Let $\homeo^+(S, \partial S)$ be the group of homeomorphisms of $S$ that preserve orientation and fix the boundary pointwise. Let $\homeo^+_0(S, \partial S)$ be the subgroup of elements homotopic to the identity via homotopies fixing the boundary pointwise.  The \emph{mapping class group} $\mcg(S)$ of $S$ is the quotient  $\homeo^+(S, \partial S)/\homeo^+_0(S, \partial S)$. Recall that $\mcg(S)$ has many equivalent definitions, for instance, by upgrading homeomorphisms to diffeomorphisms. We will introduce additional definitions as needed, deferring their presentation until required. 

 The group $\mcg(S)$ acts by permutations on the set of punctures of $S$. The kernel of this action is the \emph{pure mapping class group} of $S$ and we denote it  $\pmcg(S)$. Whenever we want to refer to an element $f\in \pmcg(S)$ as the homotopy class of a homeomorphism  $\phi \in \homeo^+(S, \partial S)$, we  write $f=[\phi]$.

Let $a$ be a curve on $S$. We denote the (right)  \emph{Dehn twist} about a curve $a$ by $d_a$. If the curve $a$ is non-separating, we say $d_a$ is a \emph{non-separating Dehn twist}.

A \emph{multitwist} $m_A \in \mcg(S)$ is a finite product of Dehn twists
\[ m_A = d_{a_1}^{n_1} \dots d_{a_k}^{n_k}, \]
where $A=\{a_1, \dots, a_k\}$ is a multicurve and the $n_i$ are non-zero integers. The curves $a_i\in A$ are the  \emph{components} of $m_A$.  We allow the set of components to be empty, in which case the multitwist is the identity. Notably, multitwists may have roots. We say $f\in \mcg(S)$ is a \emph{proper root of a multitwist} if $f$  is not a multitwist itself, but $f^k$ is  a multitwist for some $k\in \mathbb{N}$.  The \emph{degree} of the proper root $f$ is the smallest $k>0$ such that $f^k$ is a multitwist. For instance, finite order elements are proper roots of the identity. For further examples, we refer the reader to  \cite{margalit_dehn_2009} and \cite{rajeevsarathy_roots_2018}.

An  element of the mapping class group  $f\in \mcg(S)$ is said to be  \emph{reducible} if there exists a multicurve  $\eta$ such that $f$ fixes $\eta$  setwise, and $\eta$ is said to be a \emph{ reducing multicurve} for $f$. If $f$ is not reducible, we  say $f$ is  \emph{irreducible}. An element $f\in \mcg(S)$ \emph{fixes no curve} if $f(c)\ne c$  for every curve $c$. Note that fixing a curve implies reducibility, but the converse does not hold, that is, there are reducible elements that fix no curve.

Now, we extend the notions of reducibility, irreducibility and fixing no curve to homomorphisms. A map $\varphi:\pmcg(S)\to \pmcg(S')$ is \emph{reducible} if there exists a multicurve $\eta$ on $S'$ such that $\eta$ is fixed setwise by every element in the image of $\varphi$.  A homomorphism $\varphi$ is \emph{irreducible} if it is not reducible. Analogously, the homomorphism $\varphi$ \emph{fixes no curve} if for every curve $c$ on $S'$ there exists $f\in \varphi(\pmcg(S))$ such that $f(c)\ne c$. As before, a reducible homomorphism might fix a curve or fix no curve, but an irreducible homomorphism necessarily fixes no curve. 
 
 \subsection{Basic properties}
 Let $S$ be an orientable finite-type surface of genus $g\geq 1$. The Dehn-Lickorish Theorem \cite{lickorish_finite_1964} asserts that $ \pmcg(S)$ is finitely generated by Dehn twists. For surfaces of genus $g\geq 3$, we use the convenient generating set produced by Humphries  \cite{humphries_generators_1979}, namely, the set of non-separating Dehn twists $\mathcal{H}=\{d_c|\:c\in H\}$ where $H$ is the set of curves depicted in Figure \ref{fig:humphrey_gen}. In general, we say a set $\mathcal{H'}\subset \pmcg(S)$ is a \emph{Humphries generating set} if  $\mathcal{H'}$ is conjugate to $\mathcal{H}$.
 
 \begin{figure}[h]
 	\begin{center}
 		\includegraphics[width=0.7\linewidth]{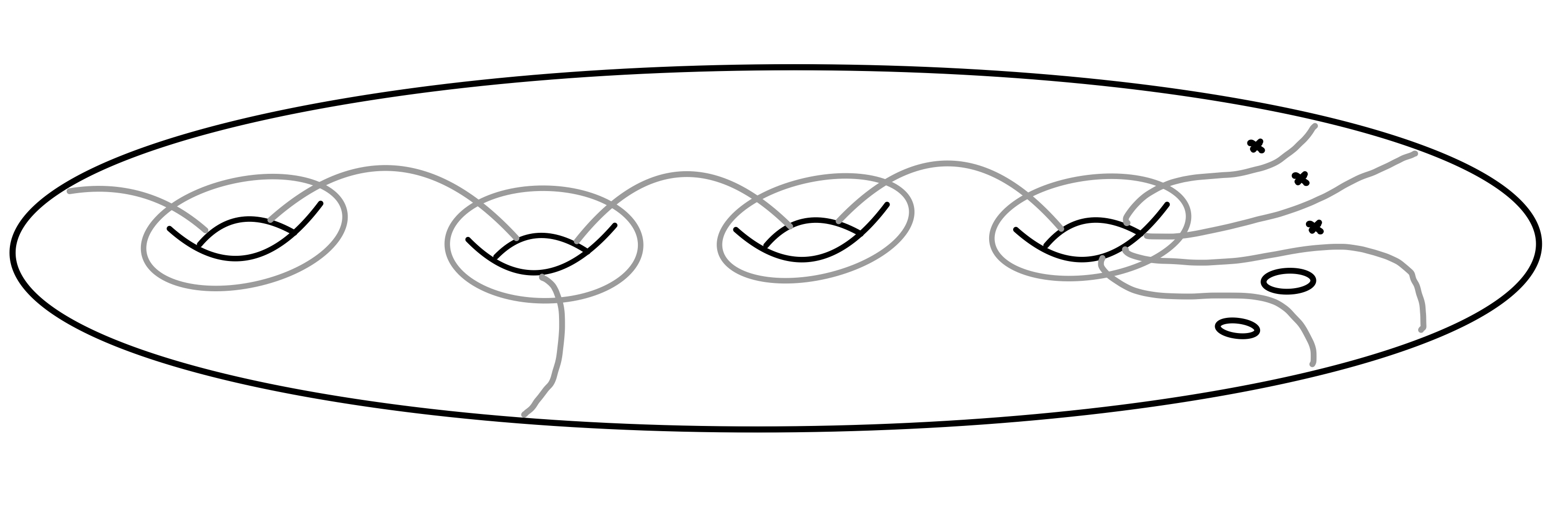}
 		\caption{Curves in a Humphries generating set for a surface of genus four, with two boundary components and three punctures.}
 		\label{fig:humphrey_gen}
 	\end{center}
 \end{figure}

Mumford \cite{mumford_abelian_1967} showed that the abelianization of $\pmcg(S)$ is finite for surfaces of genus $g\geq 2$. Powell \cite{powell_two_1978} improved on this result by showing that the abelianization is actually trivial if $g\geq 3$. \new{In general, to study the action of a group $G\subset \pmcg(S)$ on the curves of $S$ it is useful to consider the abelianization of $G$.  As an example, note that if $G$ does not surject onto $\mathbb{Z}/2\mathbb{Z}$ and fixes a curve, then $G$ fixes the curve with orientation. Similar arguments are used below and often require the theorem of Mumford-Powell, so we state it here for later reference.}

\begin{theorem}[\cite{mumford_abelian_1967}, \cite{powell_two_1978}]\label{thm:abelianization_pmcg}
	Let $S$ be an orientable finite-type surface of genus $g$. If $g=2$, then the abelianization of $\pmcg(S)$ is $\mathbb{Z}/10\mathbb{Z}$. If $g\geq 3$, then $\pmcg(S)$  has trivial abelianization.
\end{theorem}
 
 The next result we introduce follows from a combination of the Nielsen realization theorem (proved by Kerckhoff \cite{kerckhoff_nielsen_1983}) and the Hurwitz's automorphism theorem. 
 
 \begin{theorem}[\cite{kerckhoff_nielsen_1983}]\label{thm:subgrupos_finitos}
 	Let $S$ be a orientable finite-type surface of genus $g\geq 2$. Every finite subgroup $H< \pmcg(S)$ has order  $|H|\leq 84(g-1)$.
 \end{theorem}
 
 Nakajima \cite{nakajima_abelian_1987} and Maclachlan \cite{maclachlan_abelian_1965} specialized the previous theorem to  the case of finite abelian subgroups.
 
 \begin{theorem}[\cite{nakajima_abelian_1987}, \cite{maclachlan_abelian_1965}]\label{thm:subgrupos_finitos_abelianos}
 	Let $S$ be an orientable finite-type surface of genus $g\geq 2$. Then, every finite abelian subgroup  $H<\mcg(S)$ has order $|H|\leq 4g+4$.
 \end{theorem}

For later use, we record the following two properties.
 
 \begin{remark}\label{rmk:frontera_implica_no_subgrupos_finitos}
 	Let $S$ be an orientable finite-type surface with boundary, then $\mcg(S)$ has no finite order elements.  
 \end{remark}
 
 \begin{remark}\label{rmk:puntos_marcados_subgrupos_finitos_ciclicos}
 	Let $S$ be an orientable finite-type surface with at least one puncture, then every finite subgroup of $\pmcg(S)$ is cyclic. 
 \end{remark}

 \subsection{Multitwists}
 In this section we collect some facts about multitwists, which apply in particular to Dehn twists.
 
 Consider two multitwists 
 \begin{align*}
 	& m_A = d_{a_1}^{n_1}d_{a_2}^{n_2}\ldots d_{a_k}^{n_k} \\
 	& m_B = d_{b_1}^{m_1}d_{b_2}^{m_2}\ldots d_{b_{l}}^{m_l},
 \end{align*}
 where $A=\{a_1, \ldots, a_k\}$ and $B=\{b_1, \ldots, b_l\}$ are the components of $m_A$ and $m_B$, respectively. The multitwists $m_A$ and $m_B$ commute if only if the components are pairwise disjoint, i.e, $i(a,b)=0$ for every $a\in A$ and $b\in B$. 
 
 A standard result  \cite[Lemma 3.17]{farb_primer_2012} ensures that the multitwists $m_A$ and $m_B$  satisfy $m_A=m_B$ if and only if there exists a bijection $\sigma: \{1, \ldots, k\} \to \{1, \ldots, l\} $ such that 
 \[a_i = b_{\sigma(i)} \; \text{  and  } \; n_i = m_{\sigma(i)}.   \]
 In words, two multitwists are the same if and only if both have the same components with the same powers. As a consequence, for any $p\in\mathbb{Z}\setminus \{0\}$ we have that  $m_A^p =m_B^p$ if and only if $m_A=m_B$. 
 
 The conjugate of a multitwist is also a multitwist. In fact, an element $f\in \mcg(S)$ conjugates the multitwist   $m_A$ to
 \[ f \cdot m_A \cdot f^{-1}= m_{f(A)} = d_{f(a_1)}^{n_1}d_{f(a_2)}^{n_2}\ldots d_{f(a_k)}^{n_k}. \]
 Recall that $\pmcg(S)$ acts transitively on the set of non-separating curves and thus any two non-separating Dehn twists are conjugate. To simplify notation, we write in the sequel   $g^{f} =  f \cdot g \cdot f^{-1}$ for any two elements $f, g\in\mcg(S)$.
 
To conclude this section, we  present a characterization of braided multitwists proved in \cite{de_pool_braided_2024}.

 \begin{theorem}[\cite{de_pool_braided_2024}]\label{lema:trenzas}
 	Let $S$ be an orientable surface. If $m_A, m_B\in \mcg(S)$ are two multitwists satisfying the braid relation $m_Am_Bm_A = m_Bm_Am_B$, then  
 	\begin{align*}
 		m_ A&= d_{a_1}^{n_1}\dots d_{a_k}^{n_k} m_C \\ 
 		m_ B&= d_{b_1}^{n_1}\dots d_{b_k}^{n_k} m_C, \\ 
 	\end{align*}
 	where $n_i \in \{1, -1\}$, the curves $a_1,\dots, a_k,b_1,\dots, b_k$ are pairwise disjoint except for $i(a_i,b_i)=1$ and $m_C$ is a common multitwist fixing every curve $a_i$ and $b_j$.
 \end{theorem}

 \subsection{Some examples of homomorphisms}\label{sec:homs}
 
 A natural starting point to characterize maps $\varphi:\pmcg(S)\to \pmcg(S')$ is to consider  some examples. Hence, we introduce  well-known constructions of homomorphisms between pure mapping class groups. 
 
 Let $S\subset S'$ be a closed subsurface of $S$. There exists a natural homomorphism known as the \emph{inclusion homomorphism}
\begin{equation}\label{eq:hom_inclusion}
	\hat{\iota}:\pmcg(S) \to \pmcg(S').
\end{equation}
The map $\hat{\iota}$ is defined as follows: For every class $[\phi] \in \pmcg(S)$, consider the homeomorphism $\hat{\phi}:S'\to S'$ acting as $\phi$ on $S$  and acting as the identity on  $S'\setminus S$; then, the map $\hat{\iota}$ is given by the equation  $\hat{\iota}([\phi]) = [\hat{\phi}]$. A result \cite[Lemma 3.16]{farb_primer_2012} ensures that $\hat{\iota}$ is a well-defined homomorphism with kernel generated by Dehn twists $d_a$ for each curve $a\in \partial S$ bounding a puncture in $S'$  and by multitwists $d_bd_c^{-1}$,  where $b, c\in \partial S$ are homotopic curves  in $S'$.

A special case of the inclusion homomorphism is the \emph{boundary deletion map}. Let $S$ be a surface with boundary. Let $S'$ be the surface obtained by gluing to $S$  a punctured disk along a boundary curve $a \in \partial S$. The inclusion $\iota: S\hookrightarrow S'$ induces a homomorphism $\hat{\iota}:\pmcg(S)\to \pmcg(S')$ and the kernel of $\hat{\iota}$ is the group $\mathbb{Z}=\langle d_a \rangle$.

Let $S$ be a surface and $\eta$ a multicurve on $S$. The stabilizer  $\centralizer(\eta)$ of $\eta$ is the subgroup of  $\pmcg(S)$  fixing the multicurve $\eta$ setwise. We denote  $\centralizer_0(\eta) \leq \centralizer(\eta)$ the subgroup of elements that fix each curve  $c\in \eta$ with orientation. There exists a natural homomorphism from $\centralizer_0(\eta)$ obtained by `cutting' $S$ along the curves in $\eta$. To be more precise, for each curve $c\in \eta$  take $\cerrado{c}$ to be a closed regular neighborhood of $c$. The components of $S\setminus \bigcup_{c\in \eta} \cerrado{c} $ are open subsurfaces $R_i\subset S$. Then, the \emph{cutting homomorphism} is defined as the map
\begin{equation}\label{eq:hom_cortado}
	\cut_\eta: \centralizer_0(\eta) \to \prod_{i=1}^m \pmcg(R_i),
\end{equation}
where each $f\in \centralizer_0(\eta)$ is sent to  $(f_1,\dots, f_n)$ and $f_i$  is the restriction of  $f$ to the subsurface $R_i$.  The kernel of $\cut_\eta$ is the group $\mathbb{Z}^{|\eta|}$ generated by the Dehn twists $d_a$ for $a\in \eta$.
\begin{remark}
	If $f\in \centralizer_0(\eta)$ is a proper root of degree $k$ of a multitwist and the components of $f^k$ are contained in $\eta$, then $\cut_\eta(f)$ has order $k$.
\end{remark}

Let $S$ be a surface with punctures. The \emph{Birman exact sequence} is the following short exact sequence
\begin{equation}\label{eq:hom_birman}
	1 \to \pi_1(\overline{S},\{*\}) \to \pmcg(S) \xrightarrow{\forget} \pmcg(\overline{S})\to 1,
\end{equation}
where $\overline{S} \setminus\{*\}=S$. In words, $\overline{S}$ is the surface obtained by `filling' the puncture $*$ on $S$. The \emph{forgetful} map $\forget$ is defined as  $\forget([\phi]) = [\overline{\phi}]$, where $\overline{\phi}$  acts as $\phi$ on $S$ and sends the new point $*$ to itself. It is a theorem  \cite[Chapter 4.2]{farb_primer_2012} that  $\forget$ is a well-defined homomorphism and the kernel is  isomorphic to $\pi_1(\overline{S},\{*\})$.

The next  exact sequence is a consequence of the Birman-Hilden theory \cite{birman_isotopies_1973} and concerns centralizers of finite order elements (see also  \cite[Theorem 2.8]{aramayona_homomorphisms_2013}). Let $h\in \mcg(S)$  be a finite order element and denote  $\centralizer(h)$ its centralizer in $\mcg(S)$. Let $\sigma:S\to S$  be a finite order  homeomorphism representing $h$  and consider the surface  $Q$ obtained by removing the singular points of the quotient  $S/\langle \sigma\rangle$.  We have the following exact sequence
\begin{equation}\label{eq:hom_birman_hilden}
	1 \to \langle h\rangle \to \mathcal{Z}(h) \to \mcg^\pm(Q).
\end{equation}
Here $\mcg^\pm(Q)$ denotes the group of homeomorphisms of $Q$ (possibly inverting the orientation) modulo homotopy.  \new{The genus and punctures of the surface $Q$ can actually be controlled in terms of the genus of $S$:}

\begin{lemma}[{\cite[Lemma 7.3]{aramayona_homomorphisms_2013}}]\label{lemma:birman_hilden_bounds}
	\new{Let $S$ be a surface of genus $g\geq 0$ \neww{with $p\geq 0$ punctures}  and let $h: S \to S$ be a nontrivial diffeomorphism of prime order, representing an element of \neww{$\pmcg(S)$. If $Q$ is the
	surface obtained by removing the singular points from the
	quotient $S/\langle h\rangle$, then $Q$ has genus $\tilde{g}$ and $\tilde{p}$ punctures satisfying $p\leq \tilde{p}\leq 2g+2$ and 
	 $\tilde{g}\leq (2g+2-\tilde{p})/4$.}
	}
\end{lemma}

\new{Lastly, let us also record the following simple relevant fact concerning homomorphisms out of $\pmcg(S)$. }

\begin{remark}\label{rmk:trivial_homs}
	\new{
	Let $\varphi: \pmcg(S)\to G$ be a homomorphism to a group $G$. If the map $\varphi$ is nontrivial, then the image of any non-separating Dehn twist is nontrivial, since non-separating Dehn twists generate $\pmcg(S)$ and they are all conjugate to each other. }
\end{remark}

\subsection{More tools}

In this section we introduce two results regarding quotients of the mapping class group. The first one is a theorem of Berrick-Gebhardt-Paris \cite{paris_finite_2011} bounding the finite actions of $\pmcg(S)$.

\begin{theorem}[\cite{paris_finite_2011}]\label{thm:berrick_gebhardt_paris}
	Let $S$ be an orientable finite-type surface of genus $g\geq 3$. If $\pmcg(S)$ acts non trivially on a finite set $F$, then  $|F|\geq 2^{g-1}(2^g-1)$. 
\end{theorem}
\begin{remark}
	Theorem \ref{thm:berrick_gebhardt_paris} is stated in  \cite{paris_finite_2011} only for compact surfaces with (possibly empty) boundary.  The result for punctured surfaces follows from their statement. We sketch the argument.
	
	Let $S_{g,p}^b$ be the surface of genus $g$ with $b$ boundary components and $p$ punctures. Let $S_{g}^{p+b}$ be the surface of genus $g$ with $p+b$ boundary components.  The boundary deletion map induces an epimorphism 
	\[ \pmcg(S_{g}^{b+p}) \xrightarrow{\Omega} \pmcg(S_{g,p}^b). \]
	If $\pmcg(S_{g,p}^b)$ acts non-trivially on a set $F$, then $\pmcg(S_{g}^{p+b})$ acts on $F$ via the map $\Omega$. It follows from the compact case that  $|F|\geq 2^{g-1}(2^g-1)$, as we wanted. 
\end{remark}

The next result of Bridson \cite{bridson_semisimple_2010} controls the image of a Dehn twist under homomorphisms.

\begin{theorem}[\cite{bridson_semisimple_2010}]\label{thm:bridson}
	Let $S, S'$ be orientable finite-type surfaces and let $\varphi:\pmcg(S)\to \pmcg(S')$  be a homomorphism. If $S$ has genus $g\geq 3$, the image of Dehn twist is a root of a multitwist. 
\end{theorem}

We devote the rest of the preliminaries to introduce homomorphisms induced by embeddings and multi-embeddings, respectively. 

\subsection{Homomorphisms induced by embeddings}\label{subsecPrelimEmb}

Let $S$ and $S'$ be two surfaces. An \emph{embedding} from $S$ to $S'$ is a continuous injective map $\iota: S\xhookrightarrow{} S'$. Fix an embedding $\iota$ and let $\phi\in \homeo_c^+(S, \partial S)$ be a compactly supported homeomorphism. We may extend $\phi$ to a homeomorphism  $\widehat{\phi}:S'\to S'$ by letting
 \begin{equation}\label{eq:ext_phi_emb}\tag{$*$}
 	\widehat{\phi}(x) = 
 	\begin{cases}
 		x & x\not \in \iota(S),\\ 
 		\iota \circ \phi \circ \iota^{-1}(x) & x\in \iota(S).
 	\end{cases}
 \end{equation}

In this way, an embedding $\iota$ naturally induces an injective homomorphism  between the compactly supported homeomorphism groups  
\begin{align*}
	\widehat{\iota}: \homeo_c^+(S, \partial S) &\xhookrightarrow{}  \homeo_c^+(S', \partial S') \\ 
	 \phi  &\mapsto  \widehat{\phi}.
\end{align*}
 In \emph{some} cases the map $\widehat{\iota}$ descends to a homomorphism between mapping class groups, which motivates the next definition.

\begin{definition}\label{def:embed_que_induce}
	An  embedding $\iota: S \to S'$  \emph{induces a map between mapping class groups} if the map $\widehat{\iota}: \pmcg(S)\to \pmcg(S')$ given by $\widehat{\iota}([\phi])= [\widehat{\phi} \,]$ is a well-defined homomorphism, where $\widehat{\phi}$ is defined as \eqref{eq:ext_phi_emb}.
\end{definition}

A source of examples for Definition \ref{def:embed_que_induce} comes from the inclusion homomorphism. Recall that if $S$ is a subsurface of $S'$ and $S$ is closed as a subset, then there is an inclusion map $\varphi: \pmcg(S)\to \pmcg(S')$ at the level of mapping class groups. The map $\varphi$ is precisely the homomorphism induced by the natural embedding $S\to S'$.

As hinted before, not every embedding induces a map between mapping class groups. We illustrate this in the following example. 

\begin{example}\label{example:emb_no_induce_mapa}
	Let $S'$ be a compact surface of genus $g\geq 2$ and let $S$ be the \emph{interior} of a genus $g-1$ subsurface. The embedding $\iota: S \hookrightarrow S'$ given by  inclusion does \emph{not} induce a homomorphism between the mapping class groups. 
\end{example}

With this example in mind, we characterize which embeddings  induce homomorphisms. 

\begin{lemma}\label{lema:emb_que_induce}
	Let $S$ and $S'$ be two orientable surface of finite-type. An embedding  $\iota:S\hookrightarrow S'$  induces a map between mapping class groups if and only if for every puncture  $p$ of  $S$  and every simple closed curve $\gamma$ homotopic to $p$, the curve $\iota(\gamma)$ bounds a disk or a once-punctured disk.
\end{lemma}

If $\iota(\gamma)$ bounds a once-punctured disk, we say $\iota$ sends the puncture $p$ to a puncture in $S'$. If $\iota(\gamma)$ bounds a disk, we say $\iota$ sends the puncture $p$ to a disk. 

\begin{proof}
	If $\iota$  is an embedding inducing a map between mapping class groups, it follows quickly that every puncture is sent to a puncture or a disk. Thus, we  only prove the converse. 
	
	Consider the map $\widehat{\iota}: \pmcg(S) \to \pmcg(S')$ given by $\widehat{\iota}([\phi]) = [\widehat{\phi}]$. To check  $\widehat{\iota}$ is  well-defined, first  assume that $\iota: S\to S'$ sends every puncture to a puncture. In this case, $\iota(S)$ is a  subsurface of  $S'$ which is closed as a subset. Therefore, if $\phi_1, \phi_2\in \homeo_c^+(S,\partial S)$ are two homotopic homeomorphisms of $S$, then  $\widehat{\phi_1}$ and $\widehat{\phi_2}$ are homotopic homeomorphisms of  $S'$. It follows that  $\widehat{\iota}$  is a well-defined. To check $\widehat{\iota}$ is homomorphism it is enough to observe that $\widehat{\phi_1 \circ \phi_2} = \widehat{\phi_1} \circ \widehat{\phi_2}$. 
	
	Assume now that $\iota$ sends only one puncture $p$ on $S$ to a disk on $S'$. We consider two cases: either $\overline{\iota(S)} \setminus \iota(S)$ is a single point or it is a curve. In the first case,  we may factor $\iota$ as 
	\[ S\xhookrightarrow{F} \overline{S} \xrightarrow{\iota_1} S', \]
	 where $\overline{S}\setminus\{p\} = S$, and both $F$ and $\iota_1$  are the natural inclusions. Notice $\iota_1$ induces a map between mapping class groups. Indeed, since $\iota$ sends only one puncture to a disk, it follows $\iota_1$ sends every puncture to a puncture and therefore induces a map $\pmcg(\overline{S}) \to \pmcg(S')$. Moreover, $F$ induces the forgetful map $\forget: \pmcg(S)\to \pmcg(\overline{S})$. Since 	$\widehat{\iota}$  coincides with the composition  $\widehat{\iota}_1 \circ \forget$, it follows $\widehat{\iota}$  is well-defined.  Now, in the case  $\overline{\iota(S)} \setminus \iota(S)$  is a curve, we may again consider an appropriate  $\overline{S}$, $F$ and $\iota_1$  so that  $\widehat{\iota}$ is the composition  $\widehat{\iota}_1 \circ \forget$.
	
	To finish, if $\iota:S\to S'$ sends two or more punctures to disks, we may show $\widehat{\iota}:\pmcg(S)\to\pmcg(S')$ is well-defined by inductively repeating the argument in the previous paragraph.
\end{proof}

\begin{remark}\label{rmk:cerrada_homeoext_induce}
	If $S$ has no punctures, Lemma \ref{lema:emb_que_induce} implies that every embedding  $S\hookrightarrow S'$ induces a map between mapping class groups. 
\end{remark}

\subsection{Homomorphisms induced by multi-embeddings}\label{sec:prelim_varios_embs}

Let $S$ and $S'$ be two surfaces. A \emph{multi-embedding} from $S$ to $S'$ is a collection of disjoint   embeddings \[\mathcal{I}= \{\iota_i: S\xhookrightarrow{} S'|\: \iota_i(S)\cap \iota_j(S)=\emptyset \;\forall i\ne j\}.\]

Fix a multi-embedding $\mathcal{I}= \{\iota_i|\: i=1, \ldots, n\}$  from $S$ to $S'$  and let $\phi \in \homeo_c^+(S,\partial S)$ be a compactly supported homeomorphism. We may extend $\phi$ to a homeomorphism $\widehat{\phi}:S'\to S'$ by letting 

\begin{equation}\label{eq:ext_homeo_multiemb}\tag{$**$}
	\widehat{\phi}(x) = 
	\begin{cases}
		x & x\not \in \iota_i(S) \; \forall i,\\ 
		\iota_i \circ \phi \circ \iota_i^{-1}(x) & x\in \iota_i(S).
	\end{cases}
\end{equation}
Again, the multi-embedding $\mathcal{I}$ induces  an injective homomorphism 
\begin{align*}
	\widehat{\mathcal{I}}: \homeo_c^+(S, \partial S) &\xhookrightarrow{}  \homeo_c^+(S', \partial S'). \\ 
	\phi  &\mapsto  \widehat{\phi}.
\end{align*}
In \emph{some} cases $\widehat{\mathcal{I}}$ descends to a homomorphism  $\pmcg(S)\to \pmcg(S')$. The next two examples illustrate such homomorphisms.

\begin{example}\label{ex:hom_ind_por_varios_emb}
	Let $S'$ be a closed surface of genus $2g$ and let $S$ be a surface of genus $g-1$ with one boundary component. Consider two embeddings \[\iota_i: S\to S' \text{   for  }i=1,2\]  such that $\iota_1(S)\cap \iota_2(S)=\emptyset$. By Remark \ref{rmk:cerrada_homeoext_induce}, each $\iota_i$ induces a homomorphism $\widehat{\iota}_i$ between mapping class groups. Notice that elements in the  image of $\widehat{\iota}_1$  commute with elements in the image of $\widehat{\iota}_2$, since they have disjoint support. Therefore, we may consider the product homomorphism $\prod \widehat{\iota}_i: \pmcg(S)\to \pmcg(S')$ defined by $\prod \widehat{\iota}_i (f)= \widehat{\iota}_1(f) \widehat{\iota}_2(f)$. 
	
	We may also construct the homomorphism $\prod \widehat{\iota}_i$ as follows. Let   $\phi\in \homeo_c^+(S)$ be a homeomorphism and consider the extension $\widehat{\phi}:S'\to S'$ defined as in \eqref{eq:ext_homeo_multiemb}. As before,  $\widehat{\phi} \in \homeo^+_c(S')$ and  we may define the map  $\widehat{\mathcal{I}}([\phi]) = [\widehat{\phi}]$. It is now  possible to check that $\widehat{\mathcal{I}}$ is exactly the map  $\prod \widehat{\iota}_i$. 
\end{example}

In the previous example, the induced map   $\widehat{\mathcal{I}}:\pmcg(S)\to \pmcg(S')$ coincides with the product map $\prod \widehat{\iota}_i$. In general, this might not be the case. For instance, a multi-embedding $\mathcal{I}=\{\iota_1,\, \iota_2\}$ might induce a homomorphism $\pmcg(S)\to \pmcg(S')$, while each $\iota_i$ separately does not induce a homomorphism. We illustrate this case in Example  \ref{example:multi_emb_induce_no_prod}. Before doing so, we need a couple of  definitions. 

Fix two oriented surfaces $S, S'$ and let $\iota: S \to S'$ be an embedding. We say $\iota$ \emph{preserves} orientation if $\iota$ sends the orientation of $S$ to the orientation of $S'$.  Otherwise, we say $\iota$ \emph{reverses}  orientation. Additionally, we say two embeddings  $\iota_1, \iota_2 :S \to S'$  \emph{induce opposite orientations} if one of them preserves and the other one reverses the orientation. 

\begin{example}\label{example:multi_emb_induce_no_prod}
	Let $S'$ be a closed surface of genus  $2g$ and let $S$ be a once-punctured  surface with genus $g$ and no boundary. Consider two  embeddings  \[\iota_i: S\to S' \text{   for  }i=1,2\] such that  $\iota_1,\, \iota_2$ induce opposite orientations and  $\iota_1(S)\cap \iota_2(S)=\emptyset$. 
	
	Note that neither $\iota_1$ nor $\iota_2$ induce a map between mapping class groups. However, the multi-embedding $\mathcal{I}=\{\iota_1, \, \iota_2\}$  induces a map \[\widehat{\mathcal{I}}: \pmcg(S) \to \pmcg(S')\] by letting $\widehat{\mathcal{I}}([\phi]) = [\widehat{\phi}]$, where $\widehat{\phi}$ is the extension of  $\phi \in \homeo_c^+(S,\partial S)$  defined in  \eqref{eq:ext_homeo_multiemb}. For a proof that $\widehat{\mathcal{I}}$ is well-defined see  Lemma  \ref{lema:colec_emb_induc}  below. 
\end{example} 

Motivated by the previous examples, we give the following definition.

\begin{definition}\label{def:colec_emb_induc}
	The multi-embedding  $\mathcal{I} = \{\iota_i:S\hookrightarrow S'|\: i=1,2,\dots,k\}$  \emph{induces a homomorphism between mapping class groups} if the map $\widehat{\mathcal{I}}:\pmcg(S)\to\pmcg(S')$ given by $\widehat{\mathcal{I}}([\phi])= [\widehat{\phi}]$ is a well-defined homomorphism, where  $\phi \in \homeo_c^+(S,\partial S)$ and
		\[ 
	\widehat{\phi}(x) = 
	\begin{cases}
		x & x\not \in \iota_i(S) \; \forall i,\\ 
		\iota_i \circ \phi \circ \iota_i^{-1}(x) & x\in \iota_i(S).
	\end{cases}
	\]
\end{definition}

\begin{remark}
	The notion of multi-embedding induced map is analogous to the notion of `topologically diagonal map' introduced by Mann for diffeomorphism groups (see  \cite[Definition 1.4]{mann_homomorphisms_2015}).
\end{remark}

The next lemma characterizes multi-embeddings inducing homomorphisms between mapping class groups. 

\begin{lemma}\label{lema:colec_emb_induc}
	 The multi-embedding $\mathcal{I} = \{\iota_i:S\hookrightarrow S':\: i=1,2,\dots,k\}$ induces a homomorphism between mapping class groups if and only if for every curve $\gamma$ on $S$ that bounds a puncture and for every embedding $\iota_i \in \mathcal{I}$ one of the following conditions hold:
	\begin{itemize}
		\item either $\iota_i(\gamma)$ bounds a disk, 
		\item or $\iota_i(\gamma)$ bounds a once-punctured disk, 
		\item or there exists an embedding $\iota_j\in I$ such that $\iota_j(\gamma)$ is homotopic to  $\iota_i(\gamma)$ and the two embeddings induce opposite orientations.
	\end{itemize}
\end{lemma}
In the third case, when $\iota_i(\gamma)$ is nontrivial, we say $\iota_i$ sends a puncture to a nontrivial curve. 
\begin{proof}
	If $\mathcal{I}$ is multi-embedding inducing a map between mapping class groups, then it is an exercise to see that for every embedding and every puncture one of the three conditions above holds. We prove the converse. 
	
	Let $\{p_1, \dots, p_n\}$  be the set of punctures such that $p_i$ is sent to a nontrivial curve by some embedding in $\mathcal{I}$. Let $U_{p_i}$ be an open regular neighborhood for each $p_i$ and consider the subsurface  $\overline{S} =  S\setminus \bigcup_i U_{p_i}$. The surface $\overline{S}$   has one boundary component for each puncture $p_i$. We may embed $\overline{S}$ in $S'$ by composing the inclusions $\overline{S} \xhookrightarrow{s} S\xhookrightarrow{\iota_i}S$, where $s$ is the inclusion  as a subset. Now, by Lemma \ref{lema:emb_que_induce}, the embeddings $\iota_i\circ s$ induce a map between mapping class groups. Even more, if $i\ne j$ then elements in the image of $\widehat{\iota_i\circ s}$ commute with elements in the image of $\widehat{\iota_j \circ s}$. As a consequence, the set of maps $\{\iota_i\circ s\}$ induces a product map  
	\[ 
	\pmcg(\overline{S})\xrightarrow{\prod \widehat{\iota_i\circ s}}  \pmcg(S'),
	 \]
	 where $\prod \widehat{\iota_i\circ s}(f) = \widehat{\iota_1\circ s}(f) \cdot \widehat{\iota_2\circ s} (f) \cdot \ldots \cdot \widehat{\iota_{|\mathcal{I}|}\circ s} (f)$.
	 
	 On one side, the three conditions in the statement imply that  \[\prod \widehat{\iota_i\circ s}(d_a)=\text{id}\]  for every curve  $a\in \partial \overline{S}$ homotopic to a puncture $p_i$ in  $S$. On the other side, the embedding $s$ induces the `boundary deletion'  map $\widehat{s}: \pmcg(\overline{S}) \to \pmcg(S)$. From these two facts, we conclude that $\prod \widehat{\iota_i\circ s}$ factors through $\widehat{s}$, that is, there exists  $\psi$ such that the following diagram commutes
	  \begin{center}
	  	\begin{tikzcd}
 		\pmcg(\overline{S}) \arrow{d}[swap]{\widehat{s}}  \arrow{r}{\prod\widehat{\iota_i\circ s}} &\pmcg(S') \\ 
 		\pmcg(S) \arrow[dashed]{ru}[swap]{ \psi}
 		\end{tikzcd}
 	\end{center}
 	It is immediate that $\psi([\phi])=[\widehat{\phi}]$. In words,  $\psi$ is the map induced by the multi-embedding  $\mathcal{I}$ and $\psi = \widehat{\mathcal{I}}$ is a well-defined homomorphism, as we wanted. 
 \end{proof}

\section{Twist-preserving homomorphisms}\label{sec:preserva_giros}\label{appendix:prueba_tma_preserva_giros}
In this section we prove that homomorphisms $\pmcg(S)\to \pmcg(S')$ that send Dehn twists to Dehn twists are induced by embeddings. The main ideas behind the proof of this result are  contained in Aramayona-Souto \cite[Sections 9 and 10]{aramayona_homomorphisms_2013}. Here, we collect the necessary modifications to their statements and arguments. 

\begin{definition}[Twist-preserving homomorphisms]
	Let $S,S'$ be orientable finite-type surfaces and let  $\varphi:\pmcg(S) \to \pmcg(S')$ be a homomorphism. The map  $\varphi$ is \emph{twist-preserving} if the image of a non-separating Dehn twist is a nontrivial power of a Dehn twist. 
\end{definition}

\new{If the image of a single non-separating Dehn twist is a non-trivial power of a Dehn twist, then the image of any non-separating Dehn twist is a power of a Dehn twist, since they are all conjugate to each other. In particular, to check that a map is Dehn twist preserving it is enough to check it on a single non-separating Dehn twist in the domain. }

Now, we devote the rest of this section to prove the following theorem.

\begin{theorem}\label{thm:preservargiros_embinducido}
	Let $S,S'$ be orientable finite-type surface. If $S$ has genus $g \geq 3$ and   $\varphi:\pmcg(S) \to \pmcg(S')$ is a twist-preserving homomorphism, then $\varphi$ is induced by an embedding. 
\end{theorem}

\subsection{Dehn twists to Dehn twists} A priori, a twist-preserving homomorphism $\varphi$ might send a non-separating Dehn twist $d_a$ to a proper power of a Dehn twist, $\varphi(d_a)=d_b^{n}$.  The next lemma shows that $b$ must be  non-separating  and $n\in \{1, -1\}$.

\begin{lemma}\label{lemma:preserva_giros}
	Let $S, S'$ be two orientable finite-type surfaces, let $g\geq 2$ be the genus of $S$ and let  $\varphi:\pmcg(S)\to \pmcg(S')$ be a twist-preserving map. If $d_a$  is a non-separating Dehn twist, then  $\varphi(d_a)=d_b^\epsilon$ is a non-separating Dehn twist and $\epsilon \in \{1, -1\}$.
\end{lemma}
\begin{proof}
	Let $\tilde{a}$ be a non-separating curve on  $S$  such that $i(a, \tilde{a})=1$. Since $d_{\tilde{a}}$ is conjugate to $d_a$, the image satisfies
	\[ \varphi(d_{\tilde{a}}) =  \varphi(f)\cdot d_b^\epsilon \cdot \varphi(f)^{-1} =  d_{\varphi(f)\cdot b}^\epsilon = d_{\tilde{b}}^\epsilon\]
	 for some element $f\in \pmcg(S)$ and some  curve $\tilde{b}$ on $S'$.
	 
	 \new{If $b=\tilde{b}$, then $\varphi$ has infinite cyclic image. To prove this consider a curve $c$ with $i(c,a)=0$ and $i(c,\tilde{a})=1$. Then, the elements $\varphi(d_c)$ and $\varphi(d_a)=\varphi(d_{\tilde{a}})$ satisfy both  the commuting and braiding relations, which implies $\varphi(d_c)=\varphi(d_a)$. Extending the argument along a Humphries generating set yields that the image of $\varphi$ is infinite cyclic generated by $\varphi(d_a)$. However, the abelianization of $\pmcg(S)$ is finite by Theorem \ref{thm:abelianization_pmcg}, and thus  $b\ne \tilde{b}$.}
	
	Finally, since $d_a, d_{\tilde{a}}$ satisfy the braid relation, then  $d_b^\epsilon, d_{\tilde{b}}^\epsilon$ also satisfy the braid relation. From Theorem \ref{lema:trenzas} it follows that $i(b, \tilde{b})=1$ and $\epsilon \in \{1, -1\}$.
\end{proof}
\begin{remark}
	The lemma above ensures that $\varphi(d_a)=d_b^\epsilon$ with $\epsilon\in \{1, -1\}$. \new{Since any two non-separating Dehn twists are conjugate, the first paragraph of the proof above yields that $\epsilon$ is independent of the non-separating Dehn twist. In particular}, we may assume that $\epsilon=1$ up to composing by an automorphism in the image. 
\end{remark}

\new{It is worth noting that a slight modification to the proof above yields the next lemma.}

\begin{lemma}\label{lemma:exercise}
	\new{Let $S$ be a surface of genus $g\geq 3$ and let $\varphi: \pmcg(S)\to G$ be a nontrivial homomorphism to a group $G$.  If the curves $a,\tilde{a}$ \neww{satisfy $i(a, \tilde{a})\leq 1$ and} $S\setminus (a\cup \tilde{a})$ is connected, then the images $\varphi(d_a)$ and $\varphi(d_{\tilde{a}})$ are distinct.}
\end{lemma}
\begin{proof}
	\new{If the curves $a, \tilde{a}$ intersect once, then the argument is the same as in the proof of Lemma \ref{lemma:preserva_giros}. We are left to prove the case where \neww{$i(a, \tilde{a})=0$ and} $S\setminus (a\cup \tilde{a})$ is connected. In this case, there exist a set of curves $H$ such that $a,\tilde{a}\in H$, the set $\{d_c|\;c\in H\}$ is a Humphries generating set of $\pmcg(S)$ and every curve $c\in H\setminus \{a, \tilde{a}\}$  is disjoint with either $a$ or $\tilde{a}$.}
	
	\new{Seeking a contradiction, assume that $\varphi(d_a)=\varphi(d_{\tilde{a}})$. First, fix any curve $c\in H$. Without loss of generality, suppose that $c$ is disjoint with $\tilde{a}$. By the change of coordinates principle, there exists $f\in\pmcg(S)$ such that $f(a)=c$ and $f(\tilde{a})=\tilde{a}$. Now, conjugation gives that} \[\varphi(d_c)=\varphi(d_a)^{\varphi(f)}=\varphi(d_{\tilde{a}})^{\varphi(f)}=\varphi(d_{\tilde{a}}^f)=\varphi(d_{\tilde{a}})=\varphi(d_a).\]
	\new{Since $c\in H$ was arbitrary and $\{d_c|\;c\in H\}$ generate $\pmcg(S)$, then the image of $\varphi$ is infinite cyclic, but this contradicts that $\pmcg(S)$ has finite abelianization (see Theorem \ref{thm:abelianization_pmcg}}). 
\end{proof}

\subsection{Reducing to the irreducible case}

In this section we reduce the proof of Theorem \ref{thm:preservargiros_embinducido} to the case of irreducible homomorphisms. The next proposition gives sufficient conditions on an  embedding for it to induce a map between the mapping class groups (cf. \cite[Proposition 9.2]{aramayona_homomorphisms_2013}).

\begin{proposition}\label{prop:homeo_extendible_a_ind}
	Let $S,S'$ be two orientable finite-type surfaces and let $\varphi: \pmcg(S)\to \pmcg(S')$ be a homomorphism. If $\iota:S\to S'$ is an  embedding such that  $\varphi(d_a) = d_{\iota(a)}$ for every non-separating curve $a$, then $\iota$  induces the map $\varphi$.
\end{proposition}
\begin{proof}
	Let $\mathcal{H}$ be a Humphries generating set of $\pmcg(S)$, let $H=\{a|\:d_a\in \mathcal{H}\}$  and let $Z$ be a closed regular neighbourhood of $\bigcup_{a\in H}a$.  Note the surfaces $Z$ and $\iota(Z)$ have no punctures. Therefore, by Remark \ref{rmk:cerrada_homeoext_induce}, the three maps given by the restriction  $\iota|_Z: Z \hookrightarrow \iota(Z)$, the inclusion $i: Z \hookrightarrow S$ and the inclusion $i':\iota(Z)\hookrightarrow S'$ all induce homomorphisms between the corresponding mapping class groups. Even more, the induced homomorphisms satisfy 
	\[ 
	\varphi \circ \widehat{i}(d_a) = \widehat{i'} \circ \widehat{\iota|_{Z}}(d_a)\;\;\; \forall d_a\in \mathcal{H}.
	\]
	Since $\mathcal{H}$ is a generating set, the induced maps fit into the commutative diagram
	\begin{center}
		\begin{tikzcd}
			\pmcg(Z) \arrow{r}{\widehat{i}} \arrow{d}{\widehat{\iota|_Z}} & \pmcg(S) \arrow{d}{\varphi} \\
			\pmcg(\iota(Z)) \arrow{r} {\widehat{i'}}  & \pmcg(S')                    
		\end{tikzcd}.
	\end{center}
	
	The map $\widehat{\iota|_Z}$ is an isomorphism by construction. As a consequence, for every curve $c\in \partial Z$ homotopic to a puncture in $S$, we have that $\iota(c)$ bounds a puncture or a disk. Indeed, if $\iota(c)$ was not trivial, then we would have the contradiction
	\[ 
	\text{id} = \varphi \circ \widehat{i}(d_c) =  \widehat{i'} \circ \widehat{\iota|_{Z}}(d_c) = d_{\iota(c)} \ne \text{id}.
	\]
	
	It follows that  $\iota$ sends punctures in  $S$  to punctures or disks in $S'$.  By Lemma \ref{lema:emb_que_induce}, $\iota$ induces a map $\hat{\iota}$ between mapping class groups. Since $\hat{\iota}(d_a) = d_{\iota(a)} =\varphi(d_a)$ for every non-separating Dehn twist $d_a\in \mathcal{H}$, we deduce that  $\hat{\iota}=\varphi$.
\end{proof}

Let $\varphi$ be a reducible twist-preserving homomorphism. By definition, $\varphi$ fixes a reducing multicurve $\eta$. Now, next lemma shows that $\varphi$ actually fixes each curve in $\eta$ with orientation.

\begin{lemma}\label{lema:preserva_mts_masque_reducible}
	Let $S, S'$ be two orientable finite-type surfaces and let $\varphi:\pmcg(S)\to\pmcg(S')$ be a homomorphism. If $\varphi(d_c)$ is a multitwist for any non-separating Dehn twist $d_c$ and the  image of $\varphi$ fixes a multicurve $\eta$ setwise, then the image of $\varphi$ fixes each curve in $\eta$ with orientation. 
\end{lemma}
Note that, in particular, Lemma \ref{lema:preserva_mts_masque_reducible} applies to twist-preserving homomorphisms.
\begin{proof}
	Fix a Humphries generating set $\mathcal{H}$ of $\pmcg(S)$. For each $d_c \in \mathcal{H}$ denote $m_C = \varphi(d_c)$ the multitwist in the image, where  $C$ is the set of components of $m_C$. 
	
	Looking for a contradiction, consider a curve $\eta_i\in \eta$ not fixed by $\varphi(\pmcg(S))$. Since $\mathcal{H}$ is a generating set, there exists a $d_c\in \mathcal{H}$ such that $m_C(\eta_i) \ne \eta_i$. Thus, the multicurve $C$ intersects $\eta_i$. It follows that the set 
	\[ \{m_C^n(\eta_i)| \: n\in \mathbb{Z}\} \]
	is infinite. However, this contradicts that $\varphi(\pmcg(S))$ fixes the (finite) multicurve $\eta$ that contains $\eta_i$. We conclude that each curve  $\eta_i\in \eta$ is fixed by $\varphi(\pmcg(S))$.
	
	Notice that for every $d_c\in \mathcal{H}$ and $m_C=\varphi(d_c)$ we have that $i(C,\eta_i)=0$. In particular, $m_C$ fixes the orientation of $\eta_i$ and, since $\mathcal{H}$ generates,  the whole image of $\varphi$ fixes the orientation of $\eta_i$. 
\end{proof}

The next proposition reduces the proof of Theorem \ref{thm:preservargiros_embinducido} to the irreducible case. This result is parallel to \cite[Lemma 9.3]{aramayona_homomorphisms_2013} but we remark that in our setting we require a slightly different proof.

\begin{proposition}\label{prop:preserva_giros_basta_con_irreducibles}
	If  Theorem \ref{thm:preservargiros_embinducido} holds for irreducible homomorphisms, then it holds for reducible homomorphisms. 
\end{proposition}
\begin{proof}
	Let  $S$ be a surface of genus  $g\geq 3$ and $\varphi:\pmcg(S)\to \pmcg(S')$ be a reducible twist-preserving homomorphism. Consider $\eta$ a maximal reducing multicurve  for  $\varphi$. By Lemma \ref{lema:preserva_mts_masque_reducible}, the image $\varphi(\pmcg(S))$ fixes each component of  $\eta$ with orientation and so we can consider the cutting homomorphism 
	\[ \pmcg(S) \xrightarrow{\varphi} \mathcal{Z}_0(\eta) \xrightarrow{\cut_\eta} \prod_i \pmcg(R'_i), \]
	where $R'_i\subset S'$ are  surfaces such that $S'\setminus \bigcup_{\eta_i \in \eta} \eta_i = \bigsqcup_i R'_i$. 
	
	Recall that the kernel of $\cut_\eta$ is free abelian and, since $\pmcg(S)$ has trivial abelianization, the composition $\cut_\eta \circ \varphi$ is not trivial. Define \[\varphi_j = \pi_j\circ \cut_\eta \circ \varphi,\] where  $\pi_j$  is the projection to the $j$th component. Consider a subindex $j$ such that $\varphi_j$ is nontrivial. By the maximality of $\eta$, $\varphi_j$ is an irreducible twist-preserving homomorphism. From the hypotheses it follows that $\varphi_j$ is induced by an embedding $\iota: S \to R'_j$. Composing with the inclusion $i:R'_j \hookrightarrow S'$ we obtain an  embedding $ i\circ \iota: S \to S' $.  Moreover, $\varphi(d_a) = d_{i\circ \iota(a)}$ for every non-separating curve $a$. Using Proposition \ref{prop:homeo_extendible_a_ind}, we conclude that $i\circ\iota$ induces $\varphi$.
\end{proof}

\subsection{Reducing to homomorphisms that do not factor}

A homomorphism $\varphi:\pmcg(S)\to \pmcg(S')$ \emph{factors} if there is a surface $\overline{S}$, a map $\psi:\pmcg\left(\overline{S}\right) \to \pmcg(S')$ and a forgetful or boundary deletion map  $\sigma: \pmcg(S) \to \pmcg\left(\overline{S}\right)$  such that  the following diagram commutes 
\begin{center}
	\begin{tikzcd}
		\pmcg(S) \arrow{r}{\varphi} \arrow[swap]{d}{\sigma}& \pmcg(S')\\
		\pmcg\left(\overline{S}\right) \arrow[swap]{ru}{\psi}&
	\end{tikzcd}
\end{center}

\begin{lemma}\label{lema:preserva_giros_sin_factorizar}
	If Theorem \ref{thm:preservargiros_embinducido} holds for maps that do not factor, then it holds for maps that factor.  
\end{lemma}
\begin{proof}
	Proceed by induction in the number of times it factors. Note that $\varphi$ can only factor finitely many times since $S$ is of finite-type. 
	
	Assume that  $\varphi:\pmcg(S)\to \pmcg(S')$ factors only once  and let $\varphi=\psi \circ \sigma$ be such factorization. Note that the map $\sigma$ is induced by an embedding  $s$. Also, since $\varphi$ factors only once, $\psi$ does not factor. \new{Observe that every Dehn twist in $\text{Im}\sigma = \pmcg(\overline{S})$ has a lift that is a Dehn twist in $\pmcg(S)$. Then, since $\varphi$ is twist preserving, we may conclude that $\psi$ is also twist preserving}. Now, by hypothesis, $\psi$ is induced by an embedding  $p$. This implies that  $\varphi$ is induced by the embedding $p\circ s$, as we wanted. 
	
	An analogous argument can be repeated for the inductive step. 
\end{proof}

\subsection{Reducing to surfaces with no boundary}
In this section we show that if $\varphi:\pmcg(S)\to \pmcg(S')$ is an irreducible homomorphism that does not factor and satisfies the hypotheses  of Theorem \ref{thm:preservargiros_embinducido}, then $S$ and $S'$ have no boundary.

The following lemma shows that if $\varphi$ is an irreducible twist-preserving homomorphism, then $\varphi(\pmcg(S))$ has trivial centralizer. This result is similar to \cite[Lemma 9.5]{aramayona_homomorphisms_2013} but adapted to our context.  

\begin{lemma}\label{lema:preserva_giros_centralizador_trivial}
	Let $S,S'$ be two orientable finite-type surfaces, let $g\geq 3$ be the genus of $S$ and let $\varphi:\pmcg(S)\to \pmcg(S')$ be a twist-preserving homomorphism. If $\varphi$ is irreducible, then the centralizer of  $\varphi(\pmcg(S))$ is trivial. 
\end{lemma}
\begin{proof}
	First, we claim that
	\begin{claim}\label{claim:auxiliar}
		Let $a$ be a non-separating curve on $S$ and let $d_b^\epsilon = \varphi(d_a)$. If $f$ centralizes $\varphi(\pmcg(S))$, then $f(b)=b$. Even more,  $f$ preserves the orientation of $b$. 
	\end{claim}
	\begin{proof}[Proof of Claim \ref{claim:auxiliar}]
		
		Since $[f,\varphi(d_a)]=1$, then $f(b)=b$. To show $f$ fixes the orientation of  $b$, consider $a_1,a_2,a_3$ three distinct  disjoint \new{non-separating} curves on  $S$ such that $i(a, a_i)=1$ for $i=1,2,3$ \new{and such that $S\setminus (a_1\cup a_2\cup a_3)$ is connected}, \new{which exist on $S$ since it has genus $g\geq 3$}. Denote the images by   $d_{b_i}^\epsilon = \varphi(d_{a_i})$ for $i=1,2,3$. \new{By Lemma \ref{lemma:exercise}},  the curves $b, b_1, b_2, b_3$ are distinct.  Moreover, by Lemma \ref{lemma:preserva_giros}, the curves  $b_i$ are disjoint and $i(b,b_i)=1$ for $i=1,2,3$. Now, we may pick a representative  $\phi$ of $f$ and minimally intersecting representatives  $\beta, \beta_i$ of the curves $b, b_i$ such that  $\phi(\beta)=\beta$ and $\phi(\beta_i)=\beta_i$ for $i=1,2,3$. We have that $\phi$ fixes the graph  $\beta \cup \beta_1 \cup \beta_2\cup \beta_3$ and, since $\phi$ also fixes the curves, it follows that $\phi$ fixes the intersection points  $\beta \cap \beta_i$ for $i=1,2,3$. In particular, $\phi$ fixes the orientation of $\beta$, which implies $f$ fixes the orientation of $b$.
	\end{proof}
	
	Continuing with the proof of Lemma \ref{lema:preserva_giros_centralizador_trivial}, let  $\mathcal{H}$ be a Humphries generating set of  $\pmcg(S)$ and let $H=\{c|\:d_c\in \mathcal{H}\}$. We denote by  $\varphi_* (H)$ the set  
	\[ \{b|\:  d_b^\epsilon= \varphi(d_a)\text{ for }a\in H \}. \]
	
	Since $\varphi$ is irreducible, the set $\varphi_*(H)$ fills $S'$. Moreover, any two curves  $b_1, b_2 \in \varphi_*(H)$ satisfy either $i(b_1,b_2)=0$ or $i(b_1, b_2)=1$\new{,  and no three distinct curves intersect pairwise}. By the previous claim, $f$ fixes every curve  $b_i\in \varphi_*(H)$ with orientation, so the Alexander Method  \cite[Proposition 2.8]{farb_primer_2012}   yields that   $f$ is the identity. That is, the only centralizing element of the image is the identity.
\end{proof}

As a consequence of Lemma \ref{lema:preserva_giros_centralizador_trivial}, we obtain the following result.
\begin{proposition}\label{prop:irred_nofactoriza__nofronteras}
	Let $S, S'$ be two orientable finite-type surfaces, let $g\geq 3$ be the genus of $S$ and let $\varphi:\pmcg(S)\to \pmcg(S')$ be a twist-preserving homomorphism. If $\varphi$ is irreducible and does not factor, then  $\partial S = \partial S'=\emptyset$.
\end{proposition}
\begin{proof}
	If $\partial S' \ne \emptyset$, then $\varphi$ is reducible.  Thus, $\partial S' = \emptyset$.
	
	If $\varphi$ is irreducible then the image of the centralizer is trivial and we have that  $\varphi(d_a)=\text{id}$ for every curve $a\in \partial S$. Therefore, $\varphi$ factors through the boundary deletion homomorphism. Since $\varphi$ does not factor, we conclude  $\partial S=\emptyset$.
\end{proof}

Jointly,  Proposition \ref{prop:preserva_giros_basta_con_irreducibles}, Proposition  \ref{prop:irred_nofactoriza__nofronteras} and Lemma \ref{lema:preserva_giros_sin_factorizar}, imply the following simplification for the proof of Theorem  \ref{thm:preservargiros_embinducido}.

\begin{corollary}\label{cor:simplification}
	If Theorem \ref{thm:preservargiros_embinducido} holds for surfaces $\partial S = \partial S'= \emptyset$ with  irreducible maps $\varphi$ that do not factor, then Theorem \ref{thm:preservargiros_embinducido} holds in full generality. 
\end{corollary}

From here, the rest of the proof of Theorem  \ref{thm:preservargiros_embinducido} is completely analogous to the proof presented in  \cite[Section 10]{aramayona_homomorphisms_2013}, which in reality only uses that the map is twist preserving and Corollary \ref{cor:simplification}. We warn the reader that  the terminology of Aramayona-Souto is slightly different from ours: they call \emph{weak embeddings} to our embeddings and they call \emph{embeddings} to what we call embeddings inducing a map between mapping class groups.

\section{Multitwist-preserving homomorphisms}\label{sec:preserva_multigiros}
In this section we prove Theorem \ref{thm:preserva_multis_emb}.  We introduce a definition and restate the theorem here for convenience.

\begin{definition}[Multitwist-preserving homomorphism]\label{def:multitwist_preserving}
	Let $S,S'$ be two orientable finite-type surfaces and let  $\varphi:\pmcg(S) \to \pmcg(S')$ be a homomorphism. The map $\varphi$ is \emph{multitwist preserving} if the image of a non-separating Dehn twist is a non-trivial multitwist.  
\end{definition}

The name is justified since these maps send multitwists with non-separating components to multitwists. \new{Note that if the image of a single non-separating Dehn twist is a multitwist with $k$ components, then the image of any non-separating Dehn twist is also a multitwist with $k$ components, since they are all conjugate to each other. Therefore, to check that a map is multitwist preserving it is enough to check it on a single non-separating Dehn twist. In this section we show any such $\varphi$ is induced by a multi-embedding:} 

\begingroup
\def\thetheorem{\ref{thm:preserva_multis_emb}}
\begin{theorem}
		Let $S,S'$ be two orientable finite-type surfaces and let $g\geq 3$ be the genus of $S$. If   $\varphi:\pmcg(S)\to \pmcg(S')$ is a multitwist-preserving map, then  $\varphi$  is induced by a multi-embedding.
\end{theorem}
\addtocounter{theorem}{-1}
\endgroup

To prove Theorem \ref{thm:preserva_multis_emb} we may assume that the image of a non-separating Dehn twist is a multitwist with at least two components. Indeed, if the image has only one component, the statement is equivalent to Theorem \ref{thm:preservargiros_embinducido} above. 

We briefly sketch the strategy of the proof: First, we show $\varphi$ is a reducible map. Then, by cutting along a maximal reducing multicurve, we obtain a set of twist-preserving maps $\{\varphi_i\}$. Using Theorem \ref{thm:preservargiros_embinducido}, we deduce that each $\varphi_i$ is induced by an embedding $\iota_i$. To finish, we check that $\varphi$ is induced by the multi-embedding $\{\iota_i\}$. 

As  mentioned, the first objective is to show $\varphi$ is reducible.

\begin{lemma}\label{lema:preserva_mts_fija_curva}
	Let $S,S'$ be two orientable finite-type surfaces and let $g\geq 1$ be the genus of  $S$. If  $\varphi:\pmcg(S) \to \pmcg(S')$  sends non-separating Dehn twists to multitwists with at least two components, then $\varphi$ fixes a curve. In particular, $\varphi$ is reducible.
\end{lemma}
\begin{proof}
	Let $\mathcal{H}$ be a Humphries generating set for $\pmcg(S)$, let $H=\{c|\:d_c\in \mathcal{H}\}$ be the corresponding set of curves and denote by  $C$ be the set of components of the multitwist $\varphi(d_c)$ for each $d_c\in \mathcal{H}$. \neww{We use the following fact and defer its proof for later. }

	\begin{claim}\label{claim:aux_v2}
		\neww{Each connected component of $\bigcup_{d_c\in \mathcal{H}} C$  is
		graph-isomorphic either to  $\bigcup_{c\in H} c$ or to a loop. Even more, the number of connected components of $\bigcup_{d_c\in \mathcal{H}} C$ (as a graph) coincides with the number of components in the multitwist $\varphi(d_a)$, where $d_a$ is any non-separating Dehn twist in $\pmcg(S)$.}
	\end{claim}
	
	\new{ Now, consider $Z$ a closed regular neighborhood of the curves in the image $\bigcup_{d_c\in \mathcal{H}} C$. Since $Z$ has as many connected components as $\bigcup_{d_c\in \mathcal{H}} C$, which is at least two by hypothesis, then there is at least one essential curve in the boundary $\partial Z$}. Observe, $\varphi(\pmcg(S))$ fixes each connected component of $Z$ and each curve in the boundary $\partial Z$. 

	\neww{We are only left to proof the claim.}
	\begin{proof}[Proof of Claim \ref{claim:aux_v2}]
		\neww{For any curve $c_i\in H$ we denote the multitwist in the image by $m_{C_i}=\varphi(d_{c_i})$. Consider any curve $c_1\in H$ and a curve $x_1\in C_1$. We study the connected component of $\bigcup_{d_c\in \mathcal{H}} C$ containing $x_1$. }

		\neww{Choose any curve $c_2\in H$ such that $i(c_1, c_2)=1$. Since $d_{c_1}, d_{c_2}$ satisfy the braid relation (i.e, are braided), then $m_{C_1}, m_{C_2}$ are also braided. By Theorem \ref{lema:trenzas}, either $x_1\in C_2$ or there exists $x_2\in C_2$ such that $i(x_1, x_2)=1$. We treat these two cases separately. }

		\neww{First, assume that $x_1\in C_2$. Now, choose a curve $c_3\in H\setminus \{c_1, c_2\}$ that intersects either $c_1$ or $c_2$. Without loss of generality assume that $i(c_2, c_3)=1$, and so $i(c_1, c_3)=0$. Since $d_{c_2}, d_{c_3}$ are braided, so are $m_{C_2}$ and $m_{C_3}$. Then, Theorem \ref{lema:trenzas} implies that either $x_1\in C_3$ or there exists $x_3\in C_3$ such that $i(x_1, x_3)=1$. To rule out $i(x_1, x_3)=1$ just note that $d_{c_1}, d_{c_3}$ commute, which implies that $m_{C_1}, m_{C_3}$ commute and in turn $i(C_1, C_3)=0$. In particular, $x_1\in C_1\cap C_2$ does not intersect any curve in $C_3$. As a consequence, we have that $x_1\in C_3$. Repeating the argument over each curve in $x\in H$, yields that $x_1\in C$ for every curve $c\in H$. Since every $C$ is a multicurve, we deduce that $x_1$ is disjoint from every other curve in  $\bigcup_{d_c\in \mathcal{H}} C$, which implies $x_1$ is a connected component isomorphic to a loop. }

		\neww{We deal with the remaining case $i(x_1, x_2)=1$. Observe that $x_1\cup x_2$ is graph-isomorphic to $c_1\cup c_2$. We proceed by induction on the curves of $H$: Choose a curve $c_3\in H$ such that $i(c_1, c_3)=1$ or $i(c_2, c_3)=1$. Now, without loss of generality, we assume $i(c_2, c_3)=1$ and so  $i(c_1, c_3)=0$. As above, we deduce from the condition $i(c_2, c_3)=1$ that either $x_2\in C_3$ or there exists $x_3\in C_3$ such that $i(x_2, x_3)=1$. Note that $i(c_1, c_3)=0$ implies that the multitwists $m_{C_1}, m_{C_3}$ commute and in particular $i(C_1, C_3)=0$. So we conclude that $x_2\not\in C_3$, since $i(x_2, x_1)=1$ and $x_1\in C_1$. We have shown then that there exists $x_3\in C_3$ with $i(x_2, x_3)=1$ and $i(x_1, x_3)=0$. Naturally, $x_1\cup x_2\cup x_3$ is graph-isomorphic to $c_1\cup c_2\cup c_3$, and iterating the argument along the curves in $H$  yields the desired graph isomorphism between $\bigcup_{c_i \in H} x_i$ and $\bigcup_{c_i\in H} c_i$. }
	\end{proof}
	\neww{The proof of the claim completes the proof of Lemma \ref{lema:preserva_mts_fija_curva}.}
\end{proof}

\begin{remark}\label{rmk:v2}
	\neww{As it will be useful on Section \ref{sec:clasificacion}, we highlight that the proof of Claim \ref{claim:aux_v2} uses only the fact that the set of multitwists $\{m_C=\varphi(d_c)|\; c\in H\}$ satisfies the same braiding and commuting relations as the set of Dehn twists $\{d_c|\; c\in H\}$.}
	
	\neww{Also, let us note that $\bigcup_{d_c\in \mathcal{H}} C$ has no connected components isomorphic to loops if and only if the multitwists in the set $\{m_C| \;c \in H\}$ have no common components.} 
\end{remark}

Let $\varphi:\pmcg(S) \to \pmcg(S')$ be a homomorphism as in Theorem  \ref{thm:preserva_multis_emb}. Further, assume $\varphi$ sends  non-separating Dehn twists to  multitwists with at least two components. We reduce $\varphi$ to produce new maps $\varphi_i$ as follows: 

Lemma \ref{lema:preserva_mts_fija_curva} implies that $\varphi$ is reducible and so we may choose a maximal reducing multicurve $\eta$. By Lemma \ref{lema:preserva_mts_masque_reducible},  $\varphi$ fixes each curve in $\eta$ with orientation. Therefore, the image of $\varphi$ is contained in $\mathcal{Z}_0(\eta)$ and we may consider the composition 
\[ 
\pmcg(S) \xrightarrow{\varphi} \mathcal{Z}_0(\eta) \xrightarrow{\cut_\eta} \prod_i \pmcg(R'_i),
 \] 
 where the surfaces $R'_i$ are the connected components of   $S' \setminus \bigcup_{\eta_i \in \eta} \eta_i$. 
 
  For each projection $\pi_i$ to the ith coordinate, we define the map   $\varphi_i$ as 
 \begin{equation}\label{eq_def:varphi_i}
 	 \varphi_i = \pi_i \circ \cut_\eta \circ \varphi.
 \end{equation}
 Since $\varphi$ is multitwist preserving, every nontrivial $\varphi_i$ is also multitwist preserving. Moreover, such $\varphi_i$ fixes no curve by the maximality of $\eta$. So it follows from Lemma \ref{lema:preserva_mts_fija_curva} that every nontrivial $\varphi_i$ is actually twist preserving. We conclude from Theorem \ref{thm:preservargiros_embinducido} that each nontrivial $\varphi_i$ is induced by an embedding $\iota_i:S\xhookrightarrow{} R'_i$. To prove Theorem \ref{thm:preserva_multis_emb} we  show that $\varphi$ is induced by the multi-embedding $\{\iota_i: S \xhookrightarrow{} R'_i \subset S'\}$. 
 
 Before jumping into the proof, we need the next two lemmas.

 \begin{lemma}\label{lema:preservar_mults_no_comparten_curvas_1}
 	Let $S, S'$ be two orientable finite-type surfaces with $S$ of genus $g\geq 3$, let  $\mathcal{H}$  a  Humphries generating set of $\pmcg(S)$ and let $\varphi:\pmcg(S)\to \pmcg(S')$ be a multitwist-preserving map. If $d_a, d_b\in \mathcal{H}$ are two distinct Dehn twists, then  the multitwists $\varphi(d_a), \varphi(d_b)$   have no common components.
 \end{lemma}
 \begin{proof}
 	Denote the images of $d_a, d_b $ by $\varphi(d_a)=m_A$ and $\varphi(d_b)=m_B$. Also, denote the set of common components of $m_A$ and $m_B$ by $F = A \cap B$.
 	
 	\begin{claim}\label{claim_1_aux}
 		If $f\in \pmcg(S)$, then $\varphi(f)$ fixes each curve in $F$. 
 	\end{claim}
 	\begin{proof}[Proof of Claim \ref{claim_1_aux}]
 		Choose a Humphries generating set $\tilde{\mathcal{H}}$ of $\pmcg(S)$ such that $d_a, d_b \in \tilde{\mathcal{H}}$ and for every $d_c \in \tilde{\mathcal{H}}$ either $i(c,a)=0$ or $i(c, b)=0$. Now, denote the image of $d_c \in \tilde{\mathcal{H}}$ by $m_C=\varphi(d_c)$. Without loss of generality assume that $i(c, a)=0$. This implies the multitwists $m_A$ and $m_C$ commute, and so $i(A, C)=0$. In particular, $i(F, C)=0$ and therefore $m_C$ fixes each curve in $F$. \new{Since each curve in $F$ is fixed by the image of any  $d_c\in \tilde{\mathcal{H}}$} and $\tilde{\mathcal{H}}$ generates $\pmcg(S)$, the claim follows. 
 	\end{proof}
 	
 	We may assume that $i(a,b)=0$. Otherwise,  $i(a,b)=1$ and we may consider $d_c\in \mathcal{H}$ such that $i(c,a)=1$ and $i(c, b)=0$. Since there exists $f\in \pmcg(S)$ with $f(a)=b$ and $f(b)=c$, we have that $ m_A^{\varphi(f)}=m_B$ and $m_B^{\varphi(f)}=m_C$. By the previous claim, we obtain
 	\[ F=\varphi(f)(F) = \varphi(f)(A)\cap \varphi(f)(B)= B \cap C. \]
 	So, we just rename $c$ to $a$ and assume $i(a,b)=0$.  \new{Let us also remark that the complement $S\setminus (a\cup b)$ is connected, since the two Dehn twists $d_a$ and $d_b$ are contained in a Humphries generating set. }
 	
 	Choose non-separating curves $c, d, x, y, z$ such that $\{a, b, c, d, x, y, z \}$ forms a lantern (see \cite[Chapter 5]{farb_primer_2012}), that is, such that 
 	\[ d_a d_b d_c d_d = d_x d_y d_z. \]  
 	Since $S$ has genus $g\geq 3$, we may choose the curves so that the complement $S\setminus (c_1 \cup c_2)$ is connected for any two curves $c_1, c_2 \in \{a, b, c, d, x, y, z \}$. Observe that if $i(c_1, c_2) = 0$, the components in the images $m_{C_1}=\varphi(c_1)$ and $m_{C_2}=\varphi(c_2)$ \new{both contain  $F=A\cap B= C_1\cap C_2$}. To see this, consider $f\in \pmcg(S)$ such that $f(a)=c_1$ and $f(b)=c_2$, and repeat the argument in the previous paragraph. 
 	
 	For each $c_i \in \{ a, b, c, d, x, y, z  \}$ we may write the image of $d_{c_i}$ as
 	\[ \varphi(d_{c_i})=m_{C_i} = m_F m_{C_i'}, \]
 	where $m_F$ is a multitwist with components $F$, $m_{C'_i}$ is a multitwist with components $C'_i=C_i\setminus F$ and $m_{C'_i}$ commutes with $m_F$. Since any two non-separating Dehn twists are conjugate, it follows from Claim \ref{claim_1_aux} that $m_F$ is independent of the curve $c_i$.  Now, considering the lantern relation in the image we deduce that 
 	\[m_{A'} m_{B'}m_{C'}m_{D'} m_F^4= m_{X'}m_{Y'}m_{Z'} m_F^3. \] 
 	
 	Showing that  $F=\emptyset$  is equivalent to showing that $m_F=\text{id}$, which in turn is equivalent to showing that 
 	\begin{equation}\label{eq:linterna_cortada}
 		m_{A'} m_{B'}m_{C'}m_{D'} = m_{X'}m_{Y'}m_{Z'} . 
 	\end{equation}
 	
 	To prove Equation (\ref{eq:linterna_cortada}), consider  a maximal reducing multicurve $\eta$ containing $F$. Also, consider the homomorphisms $\varphi_i: \pmcg(S) \to \pmcg(R'_i)$  defined as in Equation \eqref{eq_def:varphi_i}. Recall that each $\varphi_i$ is induced by an embedding. Therefore, the image of a lantern in $S$ is a lantern in $R'_i$. By identifying the Dehn twists in $\prod_i \pmcg(R_i')$ with  Dehn twists in $\pmcg(S')$, we conclude that 
 	\[ m_{A'} m_{B'}m_{C'}m_{D'} = m_{X'}m_{Y'}m_{Z'}, \] 
 	as we wanted. 
 \end{proof}

 \begin{lemma}\label{lema:preservar_mults_no_comparten_curvas_2}
 	Let $S, S'$ be two orientable finite-type surfaces, let $g\geq 3$ be the genus of $S$ and let $\varphi:\pmcg(S)\to \pmcg(S')$ be a multitwist-preserving map.  If $d_a, d_b$ are two Dehn twists such that $i(a, b)=1$, then every component of $\varphi(d_a)$ intersects exactly  one component of $\varphi(d_b)$.
 \end{lemma}
 \begin{proof}
 	Choose a Humphries generating set $\mathcal{H}$ such that $d_a, d_b \in \mathcal{H}$. By the previous lemma, the multitwists $\varphi(d_a)$ and $\varphi(d_b)$ share no components. Moreover, $d_a, d_b$ satisfy the braid relation which implies that   $\varphi(d_a), \varphi(d_b)$  also satisfy the braid relation. It now follows directly from Theorem \ref{lema:trenzas} that every component of $\varphi(d_a)$ intersects one component of  $\varphi(d_b)$.
 \end{proof}
 
We are now ready to complete the proof of Theorem \ref{thm:preserva_multis_emb}.

\begin{proof}[Proof of Theorem \ref{thm:preserva_multis_emb}]
	Fix orientations on $S$ and  $S'$. Let $\eta$ be a maximal reducing multicurve  for $\varphi$. By Lemma \ref{lema:preserva_mts_masque_reducible}, we know that  $\varphi$  fixes each curve in  $\eta$ with orientation. Thus, we can consider the cutting map \[\cut_\eta \circ \varphi: \pmcg(S) \to  \prod_i \pmcg(R'_i),\]
	where $R'_i\subset S'$ are the subsurfaces such that $S' \setminus \eta = \sqcup_i R'_i$. Define \[\varphi_i = \pi_i \circ \cut_\eta \circ \varphi,\] where  $\pi_i$ is the projection to the ith component. Recall that  $\varphi_i$ is either trivial or twist preserving and, by Theorem \ref{thm:preservargiros_embinducido}, each nontrivial $\varphi_i$ is induced by an embedding $\iota_i: S \xhookrightarrow{} R'_i\subset S'$.

	Consider the multi-embedding $\mathcal{I}= \{\iota_i| \: \varphi_i\ne \text{id}\}$. The goal is to prove that $\mathcal{I}$ induces the map  $\varphi$. First, we check that $\mathcal{I}$ satisfies the hypotheses of Lemma \ref{lema:colec_emb_induc}.
	
	Let $n$ be a curve on $S$ homotopic to a puncture. Choose curves $a,b,c$ and $x$ on $S$ such that $i(a, b) = i(b, c)=1$, $i(a, c)=0$ and a closed regular neighborhood  of $a\cup b\cup c$ has boundary curves $n$ and $x$ (see Figure \ref{fig:chain}). Consider the multitwist 
	\begin{equation}\label{eq:construccion_embs}
		m_A = d_{\iota_1(a)}^{\epsilon_1} \dots d_{\iota_k(a)}^{\epsilon_k}, 
	\end{equation}
	where $\{\iota_1, \dots, \iota_k\} = \mathcal{I}$ and $\epsilon_i\in \{-1, 1\}$  depending on whether $\iota_i$ preserves or reverses  orientation. Analogously, we define  $m_B, \,m_C,\,m_X$ and $m_N$. Now, to check that $\mathcal{I}$ satisfies the hypotheses of Lemma \ref{lema:colec_emb_induc} it is enough to see that $m_N=\text{id}\in \pmcg(S')$.
	
	\begin{figure}[h]
		\begin{center}
			\includegraphics[width=0.5\linewidth]{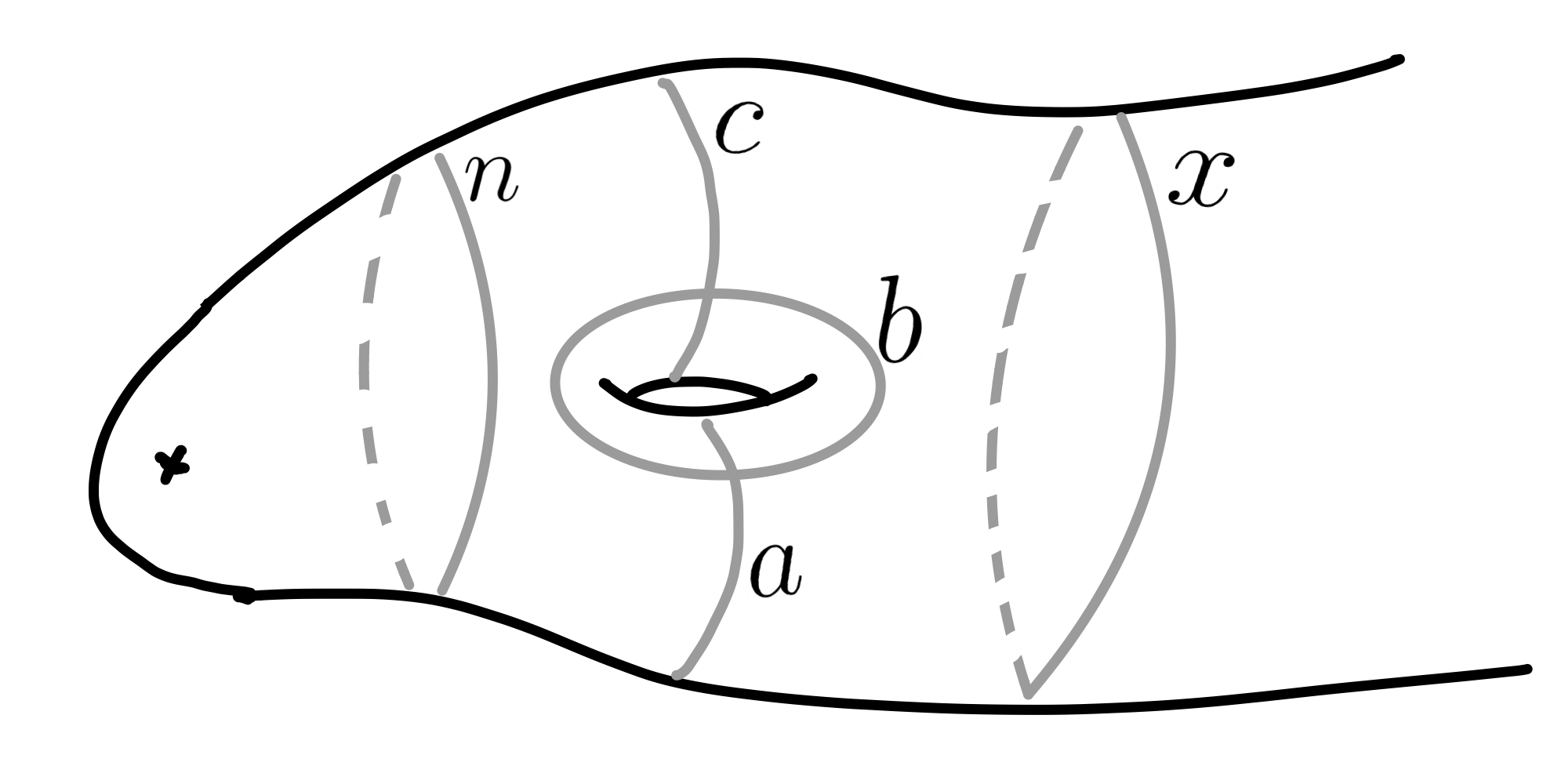}
			\caption{Curves $a, b, c, n, x$ on $S$.}
			\label{fig:chain}
		\end{center}
	\end{figure}

	By definition $\cut_\eta \circ \varphi(d_a) = m_A$. Recall that the kernel of  $\cut_\eta$  is 
	\[ \ker \cut_\eta = \langle \{d_{\eta_i}|\: \eta_i \in \eta\}\rangle.\] 
	Thus, identifying the Dehn twists in $R'_i$ with Dehn twists in $S'$ via inclusion, we have  $\varphi(d_a)=m_A \circ m$ for a multitwist $m$ with components in $\eta$.  By Lemma \ref{lema:preservar_mults_no_comparten_curvas_2},  each component of $\varphi(d_a)$ intersects some component of  $\varphi(d_b)$. However, the components of  $\varphi(d_b)$ do not intersect  $\eta$. We conclude  $m=\text{id}$ and $\varphi(d_a)=m_A$. Similarly, we deduce $\varphi(d_b)=m_B$ and $\varphi(d_c)=m_C$.  Since $x$ is a separating curve, checking that $\varphi(d_X)=m_X$ requires a little extra work. For this purpose, choose curves $\{d, e, y, z\}$ that together with $\{a, c, x\}$ form a lantern (see Figure \ref{fig:linterna}). Since $S$ has genus $g\geq 3$, we may choose the curves so that they are all non-separating except for $x$ and such that $a, c, d, e$  are disjoint and bound a four-punctured sphere. Since the curves  $d,\,e,\, y,\,z$ are non-separating, we can repeat the argument above   and conclude that  $\varphi(d_{d})=m_{D}$, $\varphi(d_{e})=m_{E}$, $\varphi(d_{y})=m_{Y}$ and $\varphi(d_{z})=m_{Z}$, where the multitwists  $m_{*}$ are also defined as in Equation \eqref{eq:construccion_embs}.
	
	\begin{figure}[h]
		\begin{center}
			\includegraphics[width=0.3\linewidth]{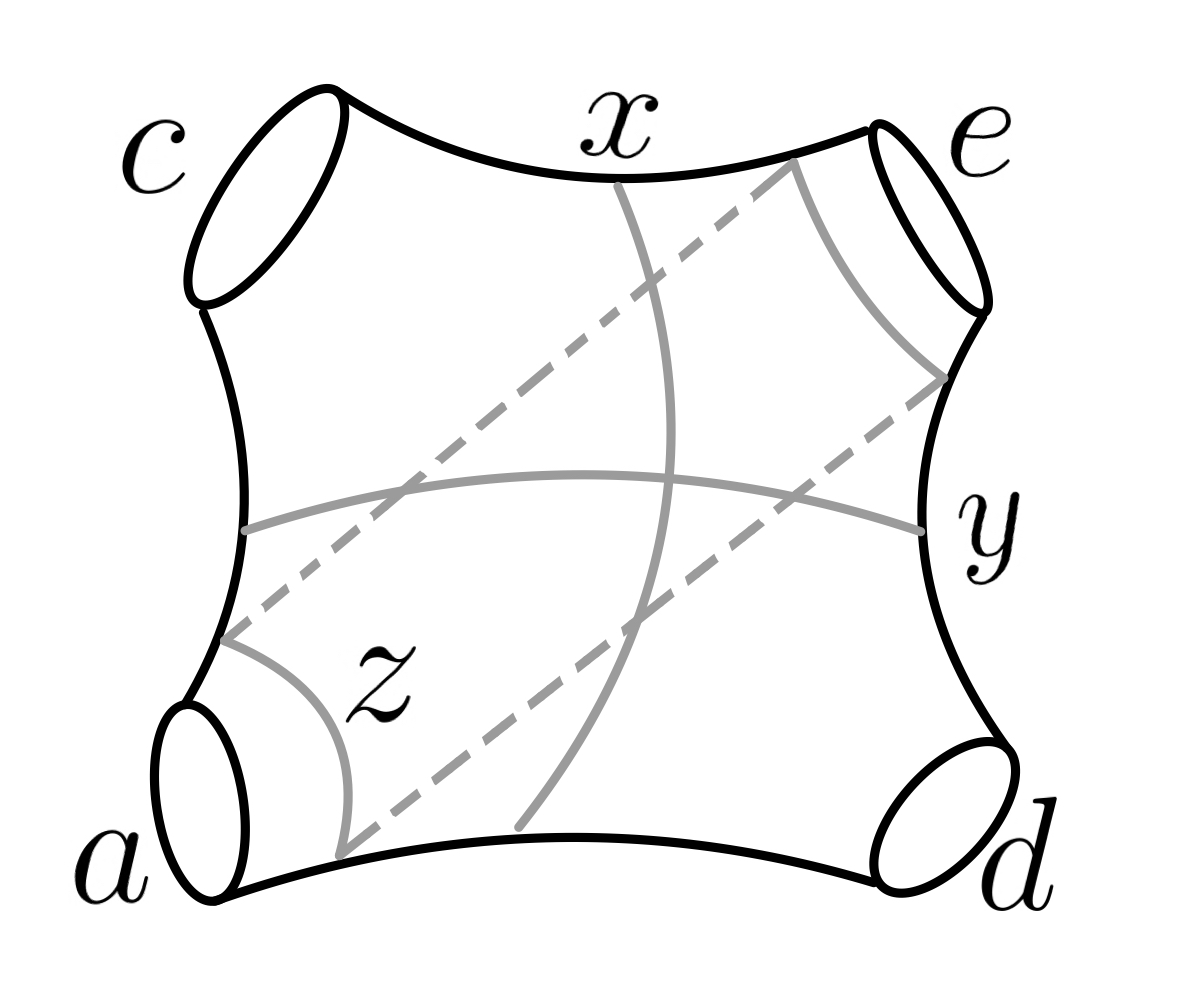}
			\caption{Curves   $a, c, d, e$ and $x, y, z$ in a lantern.}
			\label{fig:linterna}
		\end{center}
	\end{figure}
	
	Since the curves $\{a,\, c,\, d,\, e,\, x, \, y ,\, z \}$ form a lantern, we have
	\[ d_x = d_a d_c d_{d} d_{e} d_y^{-1} d_z^{-1}, \]
	and thus
	\[ \varphi(d_x) = m_A m_C m_{D} m_{E} m_Y^{-1} m_Z^{-1}.\]
	As every $\iota_i\in \mathcal{I}$ is an embedding, it follows from the last equality that $\varphi(d_x)=m_X$.
	
	Now, for each $\iota_i \in \mathcal{I}$, the Dehn twists $\{ d_{\iota_i(a)}, d_{\iota_i(b)}, d_{\iota_i(c)}, d_{\iota_i(x)}, d_{\iota_i(n)}\}$  satisfy the following relation known as the 3-chain relation (see \cite[Proposition 4.12]{farb_primer_2012})
	\[  \left(d_{\iota_i(a)}d_{\iota_i(b)}d_{\iota_i(c)}\right)^4 = d_{\iota_i(x)} d_{\iota_i(n)}. \]
	The same relation in $\pmcg(S)$ yields that  $\left( d_a d_b d_c\right)^4=d_x$. Adding these equations for every $\iota_i \in \mathcal{I}$ we obtain that 
	\begin{align*}
		m_N &= (m_A m_B m_C)^4 m_X^{-1}\\ 
		&= (\varphi(d_a)\varphi(d_b) \varphi(d_c))^4 \varphi(d_x)^{-1} \\ 
		& = \varphi((d_ad_bd_c)^4 d_x^{-1})\\ 
		&= \text{id}.
	\end{align*}
	So $m_N=\text{id}$, as we wanted.
	
	So far, we used Lemma   \ref{lema:colec_emb_induc}  to show the multi-embedding $\mathcal{I}$ induces a homomorphism  $\widehat{\mathcal{I}}:\pmcg(S)\to \pmcg(S')$. We are left to see that $\widehat{\mathcal{I}} = \varphi$. Note that for any Dehn twist $d_a\in \mathcal{H}$ in a Humphries generating set, we have that  $\varphi(d_a) = \widehat{\mathcal{I}}(d_a) \circ m$ for some $m \in \ker \cut_\eta$. Just as before, Lemma \ref{lema:preservar_mults_no_comparten_curvas_2} implies  $m=\text{id}$ and so $\varphi=\widehat{\mathcal{I}} $.
\end{proof}

\section{Classification of homomorphisms}\label{sec:clasificacion}
In this section we prove Theorem \ref{thm:clasificacion}. Notice that by Theorem \ref{thm:preserva_multis_emb}  we may restate Theorem \ref{thm:clasificacion} as follows. 

\begingroup
\def\thetheorem{\ref{thm:clasificacion}}
\begin{theorem}
	Let $S,S'$ be two orientable finite-type surfaces of genus  $g\geq 4$ and $g' \leq \cota$. If $\varphi:\pmcg(S) \to \pmcg(S')$ is a nontrivial homomorphism, then  $\varphi$ is multitwist preserving.
\end{theorem}
\addtocounter{theorem}{-1}
\endgroup

Next lemma reduces the proof of Theorem \ref{thm:clasificacion} to the case of homomorphisms fixing no curve. 

\begin{lemma}
	If Theorem \ref{thm:clasificacion} holds for homomorphisms fixing no curve, then Theorem \ref{thm:clasificacion} holds for arbitrary homomorphisms. 
\end{lemma}
\begin{proof}
	Let $\varphi:\pmcg(S)\to \pmcg(S')$ be a homomorphism fixing some curve and let $C=\{c_1,c_2,\dots , c_k\}$  be a maximal multicurve such that  $\varphi$ fixes each curve $c_i\in C$. Since $\pmcg(S)$ has trivial abelianization, the image of $\varphi$ fixes the orientation of each  $c_i\in C$. Therefore, we may consider the cutting homomorphism 
	\[\pmcg(S) \xrightarrow{\varphi} \mathcal{Z}_0(C) \xrightarrow{\cut_C} \prod_i \pmcg(R'_i),\]
	where each $R'_i$ is a connected component of  $S' \setminus \bigcup_{c_i \in C} c_i$. 
	
	We consider the maps 
	\[ \varphi_i=\pi_i \circ \cut_C \circ \varphi, \]
	 where $\pi_i$ is the projection to the $i$th component. The maximality of $C$ implies each projection $\varphi_i$ fixes no curve and, by hypotheses, each non-trivial $\varphi_i$ is multitwist preserving. Since the kernel of $\cut_C$ is generated by Dehn twists, it follows that $\varphi$ is also  multitwist preserving. 
\end{proof}

The rest of the section contains the proof of Theorem \ref{thm:clasificacion} for homomorphisms fixing no curve. The first step is to prove that the image of a Dehn twist has infinite order.

\subsection{No torsion}\label{secTorsion}
\new{ Let $\varphi:\pmcg(S)\to \pmcg(S')$ be a nontrivial homomorphism, where $S$ has genus $g\geq 3$. Although Theorem \ref{thm:bridson}  directly implies that the image of Dehn twist is a root of a multitwist, the multitwist might be trivial so the image is just a finite order element. We rule out this possibility by proving the next proposition. }

\begin{proposition}\label{propMatandoTorsion}
	Let $S, S'$ be two orientable surfaces of genus either $g\geq 4$ and $g' \leq 2^{g}-2$, or $g=3$ and $g'\leq 5$. Then, every nontrivial homomorphism $\varphi:\pmcg(S)\to \pmcg(S')$ sends non-separating Dehn twists to infinite order elements.
\end{proposition}

Proposition  \ref{propMatandoTorsion} is a consequence of the next two lemmas. 

\begin{lemma}\label{lemaGrpSimetrico}
	Let $S$ be an orientable surface of genus $g\geq 3$, let $G$ be a group and let $	\varphi:\pmcg(S)\to G$ be a homomorphism. If $\varphi(d_a)$ has order two for any  non-separating Dehn twist $d_a$, then $G$ contains a symmetric group of order $(2g+2)!$.
\end{lemma}
\begin{proof}
	Choose $2g+1$  non-separating chained curves  $c_1, c_2, \dots, c_{2g+1}$, that is, curves satisfying 
	\[ 
	i(c_i, c_j)=
	\begin{cases}
		1&  |i-j| = 1, \\
		0& |i-j|>1.
	\end{cases}
	 \]
	 
	Now, consider the subgroup 
	\[
	H = \langle d_{c_1}, d_{c_2}\dots,d_{c_{2g+1}}\rangle.  
	\]
	\new{Observe that the images of any two non-separating Dehn twists have the same order, since they are all conjugate to each other. Thus,} $\varphi(H) <G$ is a subgroup generated by elements $\varphi(d_{c_i})$ of order two. Note that the elements  $\varphi(d_{c_i})$ and $\varphi(d_{c_j})$  commute if $|i-j|>1$ and satisfy the braid relation  if $|i-j|=1$.  \new{We claim  that $\varphi(H)$ is quotient of the symmetric group of order $(2g+2)!$. Indeed, the symmetric group admits the presentation}
		\begin{align*}
			\text{Symm}_{2g+2} = \langle \tau_1, \dots, \tau_{2g+1} |\; & \tau_i^2=1  \;\;\;\;\;\;\;\;\;\;\;\;\;\;\,\forall i\\
			& \tau_i  \tau_j = \tau_j \tau_j  \;\;\;\;\;\;\;\text{ if } |i-j|>1\\
			& \tau_i \tau_j \tau_i = \tau_j \tau_i \tau_j \;\; \text{ if } |i-j|=1 \,\rangle.
		\end{align*}
	\new{Mapping each $\tau_i$ to $\varphi(d_{c_i})$ we obtain the desired quotient}. As a consequence, $\varphi(H)$ is either  $\{1\},\; \integers/2\integers$ or the entire symmetric group. Since the abelianization of $\pmcg(S)$ is trivial, it follows that $\varphi(H)\ne 1, \integers/2\integers$. Thus, $\varphi(H)<G$ is the symmetric group of order $(2g+2)!$.
\end{proof}

The next result deals with the case where the image of a Dehn twist has finite order greater than two.

\begin{lemma}\label{lema:orden_d}
	Let $S$ be an orientable surface of genus $g\geq 3$, let $G$ be a group and let $	\varphi:\pmcg(S)\to G$ be a homomorphism. If $\varphi$ sends non-separating Dehn twists to elements of finite order $d>2$, then $G$ contains an abelian subgroup $V$ of order at least  $\min (2\cdot d\cdot n^{g-1},\;d^g)$, where $n\neq1$  is some divisor  of $d$. 
	
	In particular, if $g>3$ the abelian subgroup $V$ contains at least  $2^{g+2}$ elements. If $g=3$, the group $V$ has cardinality at least $27$.
\end{lemma}

\begin{proof}
	\new{
		Choose disjoint non-separating curves $a_1,\dots, a_g$ such that $S\setminus \bigcup_{i=1}^{g} a_i$ is connected. For convenience, we
		write $v_{c}=\varphi(d_{c})$ for any curve $c$.}
	
	\new{
		Let $n\geq 1$ be the smallest natural number such that there are   $m_i\in\mathbb{Z}$  that satisfy the equation }
		\[
		v_{a_1}^n = \prod_{i=2}^g    v_{a_i}^{m_i}.
		\]
		\new{
		Clearly, $n\leq d$ since $v_{a_1}$ has order $d$. In fact, by minimality of $n$ the number $d$ is a multiple of $n$. }
		
		\new{
		Note that $n$ does not depend on the choice of curves $a_1, \dots, a_{g}$. Indeed, if $b_1, \dots, b_g$ are disjoint non-separating curves with connected complement $S\setminus \bigcup_{i=1}^{g} b_i$, then there exists $f\in \pmcg(S)$  such that $f(a_i)=b_i$ for every $i$ and  conjugating the equation above we obtain}
		\begin{align*}
			v_{b_1}^n &= \varphi(f)\cdot v_{a_1}^n\cdot \varphi(f)^{-1} \\ 
			&= \varphi(f)\cdot v_{a_2}^{m_2}\dots v_{a_g}^{m_g}\cdot \varphi(f)^{-1} \\
			&= v_{b_2}^{m_2}\dots v_{b_g}^{m_g}. \\ 
		\end{align*}
		\new{Repeating the argument with $f^{-1}(b_i)=a_i$ yields that the minimal $n$ is independent of the choice of curves, as claimed.}
		
		\new{Next, we claim $v_a^n=v_b^n$ for any two non-separating curves $a$ and $b$. To prove so, consider a  non-separating curve $a_{g+1}$ disjoint with $a_1, \dots, a_g$ such that $S\setminus \bigcup_{i=2}^{g+1} a_i$  is connected and $S\setminus (a_1\cup a_{g+1} \cup a_i)$ is connected for every $i=1, \dots, g$. There exists a map $f\in \pmcg(S)$ such that $f(a_1)=a_{g+1}$  and $f(a_i)=a_i$ for every $i=2, \dots, g$. In particular, the argument in the previous paragraph yields $v_{a_1}^n = v_{a_{g+1}}^n$.  Then, for any fixed  $j=2, \dots, g$ take a map $h\in\pmcg(S)$ such that $h(a_1)=a_j$ and $h(a_{g+1})=a_{g+1}$, so conjugating} 
		\[v_{a_j}^n= \varphi(h)\cdot v_{a_1}^n\cdot \varphi(h^{-1}) = \varphi(h)\cdot v_{a_{g+1}}^n\cdot \varphi(h^{-1}) = v_{a_{g+1}}^n.\]
		\new{It follows that $v_{a_i}^n=v_{a_j}^n$ for any two $i, j$. By extending the argument along a Humphries generating set we may show that $v_a^n=v_b^n$ holds for any two non-separating curves, as desired.}
		
		\new{If $n=1$, then $v_a=v_b$ for any two non-separating Dehn twists. This implies that the image of $\varphi$ is abelian, which contradicts that $\pmcg(S)$ has trivial abelianization for genus $g\geq 3$. We conclude that $n\geq 2$. } 
	
	Consider the finite abelian group $\tilde{V}=\langle v_{a_1},\dots,v_{a_g}\rangle$. Using the minimality of $n$, it is an exercise to see that all the elements of the form 
	\[ 
	v_{a_1}^{m}\cdot v_{a_2}^{n_2} \dots v_{a_g}^{n_g},  
	\]
	with $m\in \integers/d\integers$ and $n_i \in \integers /n\integers$, are  distinct. That is,  $\tilde{V}$ is a group of order at least $d\cdot n^{g-1}$. If $n=d$, then $\tilde{V} $ has  $d^g$ elements and we are done. It is left to consider the case  $d>n$. 
	
	We claim that if $d>n$, then there exists a non-separating  curve  $c$ of $S$ contained  in $S\setminus \bigcup_{i=1}^g a_i$ such that  $v_c \not \in \tilde{V}$. To prove it, we assume there is no such $c$.
	
	Assume that every non-separating curve  $c$ in $S\setminus \bigcup_{i=1}^g a_i$ satisfies  $v_c\in \tilde{V}$.  As $\tilde{V}$ is abelian, all the images $v_c$ pairwise commute. Now, we can take non-separating curves  $\{b, x, y, z\}$ of $S$ contained in $S\setminus \bigcup_{i=1}^g a_i$ such that $\{c_1, c_2, c_3, b, x, y, z\}$  form a lantern (see \cite[Chapter 5]{farb_primer_2012}). Considering the lantern relation in the image we have that
	\[
	v_{x}v_{y} v_{z} =  v_{a_1}v_{a_2}v_{a_3}v_{b}.
	\]
	The elements in the equation above pairwise commute as they are contained in $\tilde{V}$. So, by taking $n$-th powers, we obtain  
	\[
	v_{x}^n  =  v_{a_1}^nv_{a_2}^nv_{a_3}^n v_{b}^n v_{y}^{-n} v_{z}^{-n}. 
	\]
	
	Recall that $v_x^n = v_a^n$ for any non-separating curve $a$. Thus, the  equation above implies  $v_x^n=v_x^{2n}$ and therefore $v_x^n=\text{id}$.  However, this contradicts that   $n<d$ (recall $d$ is the order of $v_x$). We conclude that our assumption was false, i.e, there exists a non-separating curve  $c$ in  $S\setminus \bigcup_{i=1}^g a_i$  such that $v_c\not \in \tilde{V}$. Fix  any curve $c$ with this property. 
	
	To finish the case $d>n$, notice that the abelian group $\langle \tilde{V}, v_c\rangle $  has elements of the form 
	\[ 
	v_c^{\epsilon}\cdot v_{a_1}^{m}\cdot v_{a_2}^{n_2} \dots v_{a_g}^{n_g},  
	\]
	where $\epsilon \in \{0,1\}$, $m\in \integers/d\integers$ and $n_i \in \integers /n\integers$. Those elements are all distinct. Therefore, $\langle \tilde{V}, v_c\rangle$ has order at least $2\cdot d\cdot n^{g-1}$.
\end{proof}

We may now prove Proposition \ref{propMatandoTorsion}.

\begin{proof}[Proof of Proposition  \ref{propMatandoTorsion}]
	Seeking a contradiction, assume there is a non-separating Dehn twist $d_c$ whose image has order $d<\infty$.
	
	If $d=2$,  by   Lemma \ref{lemaGrpSimetrico} the group $\pmcg(S')$ has a finite subgroup of order  $(2g+2)!$. Recall that a finite subgroup of  $\pmcg(S')$ has order at most $84(g'-1)$ (see Theorem \ref{thm:subgrupos_finitos}). We deduce that  $(2g+2)! \leq 84(g'-1)$ and so 
	\[
	g'\geq \frac{(2g+2)!}{84}+1 \geq 2^g-1.    
	\]
	However, this contradicts the hypothesis $g' \leq 2^g-2$. Thus, $\varphi(d_c)$ has order  $d>2$. 
	
	If $d>2$ and $g>3$, then $\varphi(\pmcg(S))$ has an abelian subgroup of order $2^{g+2}$ (see Lemma \ref{lema:orden_d}). Recall that a finite abelian subgroup of $\pmcg(S')$ has order  at most $4g'+4$ (see Theorem  \ref{thm:subgrupos_finitos_abelianos}). As a consequence, $2^{g+2}\leq 4g' + 4$ and 
	\[  2^{g} -1 \leq g'. \] The last inequality contradicts the hypothesis  $g'\leq 2^g-2$. Thus, if  $g>3$  the element $\varphi(d_c)$ has infinite order.
	
	Lastly, if $g=3$ and $d>2$ the image of $\varphi$ has an abelian subgroup of finite order at least $27$ (see Lemma \ref{lema:orden_d}). Again, using the bound on finite abelian subgroups, we have $27 \leq 4g'+4$ which contradicts the hypothesis $g'\leq 5$. Thus, if $g=3$ then the image $\varphi(d_c)$ has infinite order. 
\end{proof}

\subsection{Counting curves}\label{secContando}

Let $S, S'$ be two surfaces as in Proposition \ref{propMatandoTorsion} and let $\varphi:\pmcg(S) \to \pmcg(S')$ be a nontrivial homomorphism. Theorem \ref{thm:bridson} and Proposition \ref{propMatandoTorsion} together guarantee the existence of an $n\in \mathbb{N}$  such that $\varphi(d_c^n)$ is a multitwist for  every non-separating curve $c$.  We write $\varphi_{*}(c)$ to denote the components of the multitwist $\varphi(d_c^n)$. In this section we compute an upper bound on the cardinality of $\varphi_{*}(c)$. 
 
 It is immediate that $\varphi_{*}$ satisfies the following properties:
\begin{itemize}
	\item If $i(c,c')=0$, then  $i(\varphi_*(c), \varphi_*(c'))=0$. 
	\item The cardinality  $|\varphi_*(c)|$ is independent of the non-separating curve $c.$
	\item $\varphi(f) \cdot \varphi_*(c) = \varphi_*(f(c))$ for every $f\in \pmcg(S)$.
\end{itemize}
The next result guarantees that forgetting punctures does not change the cardinality of $\varphi_{*}(c)$. The result was originally stated in  \cite[Lemma 6.4]{aramayona_homomorphisms_2013} with different hypotheses, however their proof remains valid for the version we state here. 

\begin{lemma}\label{lema:olvidar_no_reduce_curvas}
	Let $S, S'$ be two orientable surfaces of genus either $g\geq 4$ and $g' \leq 2^g-2 $, or  $g=3$ and $g'\leq 5$. If $\varphi:\pmcg(S)\to \pmcg(S')$ is a homomorphism fixing no curve, then the composition with the forgetful map satisfies $|(\forget \circ \varphi)_*(c)|=| \varphi_*(c)|$.
\end{lemma}
\begin{proof}
	The statement follows from the proof of \cite[Lemma 6.4]{aramayona_homomorphisms_2013}.
\end{proof}

Before computing the upper bound, we need the following two lemmas. 

\begin{lemma}\label{lema:si_fijas_curvas_con_orientacion_fijas_imagen_de_esa_curva_con_orient}
	Let $S, S'$ be two orientable surfaces of genus either $g\geq 5$ and $g' \leq 2^g-2 $, or $g=4$ and $g'\leq 10$.  If $\varphi:\pmcg(S)\to \pmcg(S')$ is a homomorphism fixing no curve and $c$ is a non-separating curve on $S$, then $\varphi(\mathcal{Z}_0(c))$ fixes every curve in $\varphi_*(c)$ with orientation.
\end{lemma}
\begin{proof}
	 Since $\varphi$ fixes no curve, $S'$ has no boundary.
	 
	We compute a rough bound for  $|\varphi_{*}(c)|$.  By Lemma \ref{lema:olvidar_no_reduce_curvas}, we can assume $S'$ is closed. \new{Provided that $g\geq 5$, then }
	\[ |\varphi_{*}(c)| \leq 3g'-3 \leq 3\cdot 2^g-9. \]
	\new{For the $g=4$ case, the first inequality will suffice.} Now, let $S_c$ be the complement of an open regular neighborhood of $c$. The inclusion $S_c \xhookrightarrow{} S$ induces an epimorphism  $\pmcg(S_c) \xrightarrow{\iota} \mathcal{Z}_0(c)$. The image of $\varphi\circ \iota$  fixes the multicurve   $\varphi_{*}(c)$ setwise and so induces an action of $\pmcg(S_c)$ on $\varphi_{*}(c)$. Since $S_c$ has genus $g-1$,  Theorem \ref{thm:berrick_gebhardt_paris} implies the action is trivial. In other words, $\varphi(\mathcal{Z}_0(c))$  fixes every curve in  $\varphi_{*}(c)$.
	
	Since $\pmcg(S_c)$ has trivial abelianization, it follows that $\varphi(\mathcal{Z}_0(c))$   fixes the orientation of each curve in $\varphi_{*}(c)$.
\end{proof}

\begin{lemma}\label{lema:debilmenteirred_no_comparten_comp}
	Let $S, S'$ be two orientable surfaces of genus either $g\geq 5$ and $g' \leq 2^g-2 $, or  $g=4$ and $g'\leq 10$.  If $\varphi:\pmcg(S)\to \pmcg(S')$ is a homomorphism  fixing no curve and $c,c'$ are two disjoint non-separating curves such that $S\setminus (c\cup c')$ is connected, then $\varphi_{*}(c)$ and $\varphi_{*}(c')$ share no curves.
\end{lemma}
\begin{proof}
	Since $\varphi$ fixes no curve,  $S'$ has no boundary. Also, by Lemma \ref{lema:olvidar_no_reduce_curvas}, we can assume $S'$ has no punctures.  Write $F = \varphi_{*}(c) \cap \varphi_{*}(c')$.
	
	 First, we claim that if $d$ is a non-separating curve on $S$ with $i(d, c)=0$ and such that $S\setminus (d\cup c)$ is connected, then $F\subset \varphi_{*}(d)$. Indeed, consider $f\in \pmcg(S)$ such that $f\in \centralizer_0(c)$ and $f(c')=d$. We know that 
	 \[ \varphi(f)\cdot \varphi_{*}(c')=\varphi_{*}(f(c'))=\varphi_{*}(d). \] 
	Now, Lemma \ref{lema:si_fijas_curvas_con_orientacion_fijas_imagen_de_esa_curva_con_orient} implies  $\varphi(f)$ fixes each curve in $F\subset \varphi_{*}(c)$. In particular, $F\subset \varphi_{*}(d)$. 
	
	We claim $F\subset \varphi_{*}(e)$, for every $e$  non-separating curve on $S$ with $i(c, e)=1$. As $g\geq 4$, there exists a non-separating curve $d$  on $S$ with $i(d, c)=i(d, e)=0$ and such that both $S\setminus (d\cup c)$ and $S\setminus (d\cup e)$ are  connected. Then,  considering $f\in  \centralizer_0(d)$ such that $f(c)=e$, we may repeat the argument above to deduce that $F\subset \varphi_{*}(e)$. 
	
	From the last two paragraphs it follows there is a Humphries generating set $\mathcal{H}$ of $\pmcg(S)$ with the property that $F\subset \varphi_{*}(c)$ for every  element $d_c \in \mathcal{H}$. Observe that for  every $d_c\in \mathcal{H}$ there exists $d_{b}\in \mathcal{H}$ with  $d_c \in \centralizer_0(b)$. Combining these facts with Lemma \ref{lema:si_fijas_curvas_con_orientacion_fijas_imagen_de_esa_curva_con_orient}, we deduce that every image $\varphi(d_c)$ fixes each curve  $x\in F$. Since $\mathcal{H}$ generates, we conclude $\varphi(\pmcg(S))$ fixes each curve  $x\in F$. By hypothesis, $\varphi$ fixes no curve, thus $F=\emptyset$. 
\end{proof}

The following proposition gives the upper bound for $|\varphi_*(c)|$.

\begin{proposition}\label{prop:contando_curvas}
	Let $S, S'$ be two orientable finite-type surfaces of genus either $g\geq 5$ and $g' \leq 2^g-2 $, or  $g=4$, and $g'\leq 10$. If $\varphi:\pmcg(S)\to \pmcg(S')$  is a homomorphism fixing no curve, then $|\varphi_*(c)|\leq \frac{g'-1}{g-1}$.
\end{proposition}
\begin{proof}
	Choose $a_1,\dots, a_{3g-3}$ disjoint non-separating curves such that each $S\setminus (a_i \cup a_j)$ is connected for any $i,j$. From Lemma \ref{lema:debilmenteirred_no_comparten_comp}, we deduce that $\bigcup_{i=1}^{3g-3} \varphi_{*}(a_i)$ has $K(3g-3)$ distinct curves, where $K=|\varphi_{*}(c)|$. Since we may assume $S'$ is closed (see Lemma \ref{lema:olvidar_no_reduce_curvas}), we have that \[ K(3g-3) \leq 3{g'}-3 \]
	and the statement follows. 
\end{proof}

\subsection{Dehn twists to multitwists}\label{secDehnsAMultis}

In this section we complete the proof of Theorem \ref{thm:clasificacion}. So far, we showed that the image of a non-separating Dehn twist is an infinite order element, root of a multitwist and that the associated multicurve has at most $\frac{g'-1}{g-1}$ components. The last step is to show that the image cannot be a proper root of a multitwist, and hence it is actually a multitwist. \new{The argument is by induction on the genus $g$, by extracting information from the inclusion homomorphism of smaller genus subsurfaces. }

\begin{proposition}{(Base case)}
	Let $S,S'$ be two orientable finite-type surfaces of genus  $g\geq 4$ and $g'\leq 6$, respectively. If $\varphi:\pmcg(S)\to \pmcg(S')$ is a homomorphism fixing no curve, then $\varphi$ is multitwist preserving. 
\end{proposition}
\begin{proof}
	Let $c$ be a non-separating curve on $S$. Looking for a contradiction,  assume  $\varphi(d_c)$ is a proper root of a multitwist of degree $k>1.$
	
	By Proposition  \ref{prop:contando_curvas}, we have that  $|\varphi_*(c)|\leq \frac{g'-1}{g-1}<2$ and so $\varphi(d_c^k)$ is a power of a Dehn twist. The curve $\varphi_{*}(c)$ is either separating or non-separating; we treat these two cases separately. 
	
	If $\varphi_{*}(c)$ is non-separating, then $R' = S' \setminus\varphi_{*}(c)$ is a connected surface. Let ${S_c}\subset S$ be the complement of an open regular neighborhood of $c$. The inclusion $S_c \xhookrightarrow{} S$  induces an epimorphism  $\pmcg(S_c) \xrightarrow{\iota} \mathcal{Z}_0(c)$.  By Lemma \ref{lema:si_fijas_curvas_con_orientacion_fijas_imagen_de_esa_curva_con_orient}, we may consider the cutting homomorphism \[\gamma=\cut_{\varphi_{*}(c)} \circ \varphi \circ \iota:  {\pmcg({S_c})} \to \pmcg(R').\]
	Notice that $R'$ is a surface with at least two punctures and of genus at most five. Also, the image of $\gamma$ commutes with the element $h=\gamma(d_c)$ of order $k$.  Up to taking a power of $h$, we can assume $h$ has primer order. Thus, we can consider the Birman-Hilden map defined  in Equation (\ref{eq:hom_birman_hilden}): Let $Q'$ be the surface obtained by removing the singular points from the quotient $R'/\langle h\rangle$ and let $\psi$ be the Birman-Hilden map induced by $h$, we have that
	\[ \pmcg(S_c) \xrightarrow{\gamma} \centralizer(h) \xrightarrow{\psi}\mcg^\pm(Q').  \]
	Recall that the kernel of  $\psi$ is $\langle h\rangle$. Since $\pmcg({S_c})$ has trivial abelianization as $g\geq 4$, the composition  $\psi \circ \gamma$ is a nontrivial homomorphism. Even more, the image of $\psi\circ \gamma$ is contained in $\mcg(Q')$.

	\new{Since $R'$ has genus at most five and at least two punctures, Lemma \ref{lemma:birman_hilden_bounds} imply} the surface $Q'$ has genus $\tilde{g} < 3$ and  $\tilde{p} \leq 12$ punctures.  By Theorem \ref{thm:berrick_gebhardt_paris}, the image of  $\psi \circ \gamma$ fixes each puncture in  $Q'$. This implies that 
	\[\psi\circ \gamma: \pmcg({S_c}) \rightarrow \pmcg(Q') \]
	 is a nontrivial map that reduces the genus of the involved surfaces, however this contradicts \cite[Proposition 7.1]{aramayona_homomorphisms_2013}.
	
	We are left with the case where $\varphi_*(c)$ is a separating curve. We denote $R'_1$ and $R'_2$ the connected components of $S'\setminus\varphi_{*}(c)$. Lemma \ref{lema:olvidar_no_reduce_curvas} implies that $\varphi_{*}(c)$  does not bound  a punctured disk, and so the surfaces $R'_1, R'_2$ both have positive genus at most five. Moreover, both surfaces $R'_1, R'_2$ have at least one puncture \new{coming from the deleted curve $\varphi_*(c)$}. As before, consider the composition  
	\[ \pmcg(S_c) \xrightarrow{\iota} \mathcal{Z}_0(c) \xrightarrow{\cut_{\varphi_{*}(c)} \circ \varphi} \pmcg(R'_1) \times \pmcg(R'_2). \]
	We denote the projections   $\gamma_i=\pi_i\circ \cut_{\varphi_{*}(c)} \circ \varphi  \circ \iota $ for $i=1,2$, where $\pi_1, \pi_2$ are the projections to the first and second component. Note that at least  one of the two projections  is nontrivial. Without loss of generality, assume $h=\gamma_1(d_c)$ is nontrivial. We have that  $ h=\gamma_1(d_c)$  is an element of finite order $k'>1$ and commutes with the image of $\gamma_1$. As in the previous paragraph, we have a map 
	\[\pmcg({S_c}) \xrightarrow{\gamma_1} \centralizer_0(h) \xrightarrow{ \psi} \mcg(Q'),\]
	where $Q'$ is the surface obtained by removing the singular points from the quotient $R'_1/\langle h\rangle$.
	
	Again,  we know  $\psi \circ \gamma_1$ is a nontrivial map.  Also, by  Theorem \ref{thm:berrick_gebhardt_paris} the image of $\psi \circ \gamma_1$ is contained in $\pmcg(Q')$, since the number of punctures of   $Q'$ is $\tilde{p}\leq 12$ (\new{see Lemma \ref{lemma:birman_hilden_bounds}}). \new{As before, Lemma \ref{lemma:birman_hilden_bounds} imply} the surface  $Q'$ has genus at most  two. We conclude that $\psi \circ \gamma_1$  is a nontrivial homomorphism between pure mapping class groups reducing the genus of the involved surfaces. Again,  this contradicts \cite[Proposition 7.1]{aramayona_homomorphisms_2013}. 
\end{proof}

Before finishing the proof of Theorem \ref{thm:clasificacion} with the inductive step, we state the last lemma we  require.

\begin{lemma}\label{lemma:orden_finito_reducible}
	Let $X$ be an orientable finite-type surface, let  $C=\{c_1,c_2 ,\dots\}$ be a multicurve on $X$ and let $Y_i$  for $i=1, \dots, m$ be the connected components of $S\setminus C$. If $h\in \pmcg(X)$ is an element of finite order $k\in \mathbb{N}$ and $h$ fixes each curve  $c_i$ with orientation, then  $h|_{Y_i}$ induces an element of order $k$ in $\pmcg(Y_i)$ for every $i=1, \dots, m$.
\end{lemma}
\begin{proof}
	We may realize  $h$ by an isometry $\phi$ of a hyperbolic structure in $X$ (see \cite[Theorem 5]{kerckhoff_nielsen_1983}). Realizing the curves in $C$ by geodesics, $\phi$ restricts to each component  $Y_i$ as a finite order element. Note the isometry  $\phi$ is locally nontrivial, since the only locally trivial isometry  is the identity. Therefore, $\phi|_{Y_i}$ is a diffeomorphism of finite order in each $Y_i$ and so induces a non-trivial element of finite order in each $\pmcg(Y_i)$. Since the argument can be repeated for any nontrivial power of $\phi$, we conclude that the diffeomorphism  $\phi|_{Y_i}$ has order $k$ and hence also does the induced mapping class. 
\end{proof}

Finally,  we complete the proof of Theorem  \ref{thm:clasificacion}.

\begin{proof}[Proof of Theorem \ref{thm:clasificacion} (inductive step)] 
	Let $\varphi: \pmcg(S)\to \pmcg(S')$ be a homomorphism fixing no curve, where $S, S'$ are finite-type orientable surfaces of genus $g\geq 5$ and $g' \leq \cota$, respectively. 
	
	Let $c_1$ be a non-separating curve in $S$ and let $S_{c_1}$ be the complement of an open regular neighborhood of $c_1$. The inclusion $S_{c_1}\xhookrightarrow{}S$ induces the epimorphism  $ \iota: \pmcg(S_{c_1}) \to \mathcal{Z}_0(c_1)$. By Lemma \ref{lema:si_fijas_curvas_con_orientacion_fijas_imagen_de_esa_curva_con_orient},    $\mathcal{Z}_0(c_1)$  is a subgroup of  $\varphi(\mathcal{Z}_0(\varphi_*(c_1)))$. Therefore, we may consider the following composition of maps
	\[ \pmcg({S_{c_1}})  \xrightarrow{\varphi\circ \iota} \mathcal{Z}_0(\varphi_*(c_1))\xrightarrow{\cut_{\varphi_*(c_1)}} \prod_i \pmcg(R'_i), \]
	where $R'_i$ are the connected components of $S\setminus  \varphi_*(c_1)$. 
	
	We define the maps 
	\[ \varphi_j=\pi_j \circ \cut_{\varphi_*(c_1)} \circ \varphi \circ \iota, \]
	where $\pi_j$ is the projection to the $j$th component. If all the projections $\varphi_j$ are multitwist preserving, then $\varphi$ is also multitwist preserving.  Seeking a contradiction, we assume that at least one projection $\varphi_j$ is not multitwist preserving. 
	
	First, suppose there are two distinct projections  $\varphi_j , \varphi_{j'}$ that are not multitwist preserving. \new{As the surface $S_{c_1}$ has genus $g-1$,  induction on the genus implies that} both surfaces $R'_j$ and $R'_{j'}$ have genus greater than $6\cdot 2^{g-5}$. It follows that  $g' >6 \cdot 2^{g-4}$, which contradicts the hypothesis of the statement. 
	
	To complete the proof, we rule out the case where there is a unique projection $\varphi_j$ that is not multitwist preserving. Without loss of generality, we  assume that $\varphi_1$ is the unique non-multitwist preserving projection.

	\begin{figure}[h]
		\begin{center}
			\includegraphics[width=0.6\linewidth]{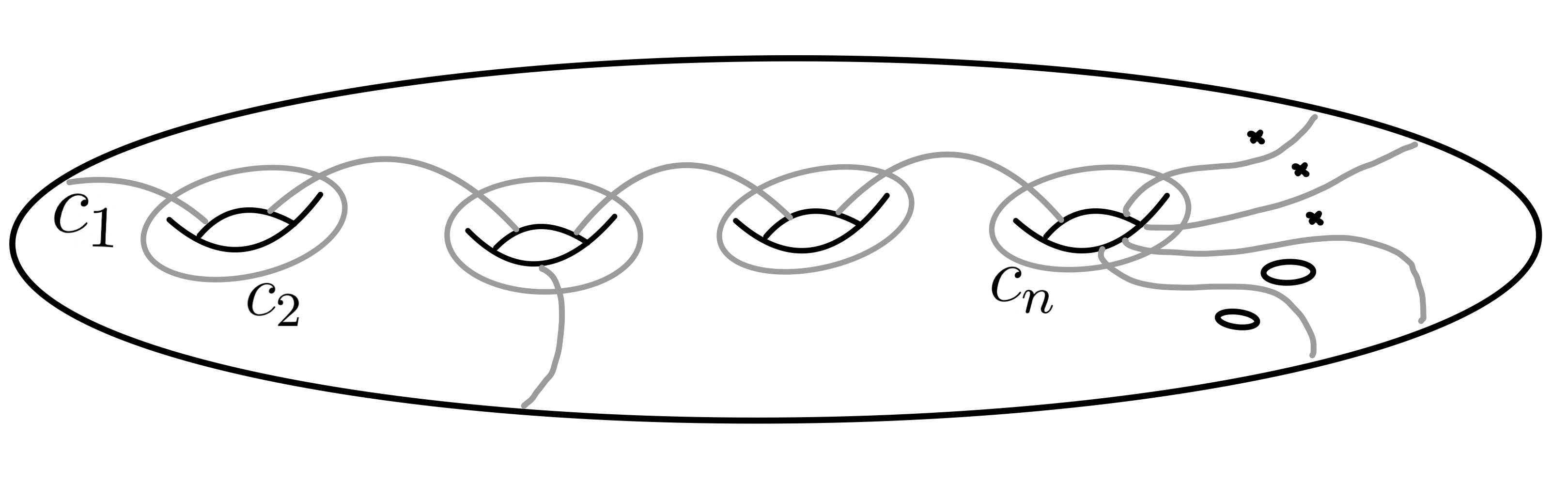}
			\caption{The set of curves $H$.}
			\label{fig:curves}
		\end{center}
	\end{figure}
	
	The contradiction follows easily from the next two claims. Choose $\mathcal{H}$ a Humphries generating set with $d_{c_1}\in \mathcal{H}$ and let $H=\{c|\: d_c\in \mathcal{H}\}$.
	
	\begin{claim}\label{claim:descomposicion}
		Consider $c_i\in H$. We can decompose $\varphi(d_{c_i})$ as 
		\[ \varphi(d_{c_i}) = M_i h_i, \]
		where $M_i$ is a multitwist with components contained in $\varphi_{*}(c_i)$ and $h_i$ is a finite order element \new{with $h_i^k=\text{id}$} that fixes every curve  $x\in \varphi_{*}(c_j)$ for every $c_j\in H$.
	\end{claim}

	\begin{claim}\label{propiedad:los_elementos_se_trenzan}
		Consider $c_i, c_j \in H$. If $i(c_i, c_j)=1$, then $h_ih_jh_i = h_jh_ih_j$ and $M_iM_jM_i = M_jM_iM_j$. 
	\end{claim}
	We assume both claims and defer their proof for later.

	Note the finite order element $h_i$ fixes the curves in $\varphi_{*}(c_i)$. In particular, $h_i$ commutes with $M_i$, which implies that $\varphi(d_{c_i}^k)=M_i^k$. Thus, the components of the multitwist $M_i$ are precisely the curves $\varphi_{*}(c_i)$.

	By \neww{Claim \ref{propiedad:los_elementos_se_trenzan}, Lemma \ref{lema:debilmenteirred_no_comparten_comp} and Remark \ref{rmk:v2}}, \new{we have a graph isomorphism between $\bigcup_{c_i\in H} c_i$ and each connected component of $\bigcup_{c_i\in H} \varphi_{*}(c_i)$. Let $Z$ and $Z'$ be closed regular neighborhoods of $\bigcup_{c_i\in H} c_i$ and $\bigcup_{c_i\in H} \varphi_{*}(c_i)$, respectively. \neww{It follows from Theorem \ref{lema:trenzas} that any two curves in $\bigcup_{c_i\in H} \varphi_{*}(c_i)$ are disjoint or (transversely) intersect exactly once, thus} the graph isomorphism induces a homeomorphism between $Z$ and each connected component of $Z'$.}
	 
	 If the curves $\bigcup_{c_i\in H} \varphi_{*}(c_i)$  fill $S'$, then $Z'$ has a single connected component. \new{Observe that the curves $\bigcup_{c_i\in H} \varphi_{*}(c_i)$ either intersect once or are disjoint, and no three distinct curves pairwise intersect}.  Therefore, the Alexander method \cite[Proposition 2.8]{farb_primer_2012} yields that $h_1$ is trivial and so $\varphi$ is multitwist preserving as $\varphi(d_{c_1}) = M_1$.  
	 
	 If the curves $\bigcup_{c_i\in H} \varphi_{*}(c_i)$ do not fill $S'$, then there exists a curve $x$ on the boundary $\partial Z'$ which is essential on $S'$. Since $h_1$ fixes each curve in $\varphi_{*}(c_1)$, then $h_1$ fixes each component  of $Z'$.  Even more, the Alexander method implies that $h_1$ \new{is trivial restricted to each component of $Z'$} and so it fixes the boundary curve $x$. Repeating the argument for every $h_j$, we deduce that any $h_j$ also fixes the curve $x$. It follows that the image of $\varphi$ fixes the curve $x$, but this contradicts that $\varphi$ fixes no curve.  
	 
	 We are only left to prove the claims.

	 \begin{proof}[Proof of Claim \ref{claim:descomposicion}]
	 	
	 	Let $c_2$ be a curve on $S$ with $i(c_1, c_2)=1$. Since $g\geq 5$, we may choose a curve $c_n$ on $S$ with $i(c_1, c_n)=i(c_2, c_n)=0$ and such that both $S\setminus(c_1\cup c_n)$ and $S\setminus (c_2 \cup c_n)$ are connected (see Figure \ref{fig:curves}).

	 	Note the element $\varphi_1(d_{c_1})$ has order $k$, where $k$ is the degree of the proper root $\varphi(d_{c_1})$.  Indeed, Theorem  \ref{thm:preserva_multis_emb} implies that all nontrivial projections  $\varphi_i$ with $i\ne 1$ are induced by multi-embeddings and so the elements   $\varphi_i(d_{c_1})$ are either trivial or have infinite order. Since    $\cut_{\varphi_*(c_1)} \circ \varphi (d_{c_1})$ has finite order, the elements $\varphi_i(d_{c_1})$ are trivial for  $i\ne 1$. It follows $\varphi_1(d_{c_1})$ has order $k$.
	 	
	 	Now, we cut the surface $R'_1$ along the curves in ${\varphi_{1}}_*(c_n)$. Denote by ${S_{c_1\cup c_n}}$ the complement to an open regular neighborhood of  $c_1\cup c_n$. The inclusion $S_{c_1\cup c_n} \to S_{c_1}$ induces a homomorphism $\sigma: \pmcg(S_{c_1\cup c_n}) \to \pmcg(S_{c_1})$. Since the restriction $(\varphi_1 \circ \sigma )|_{\pmcg({S_{c_1\cup c_n}})}$ is contained in  $\mathcal{Z}_0({\varphi_{1}}_*(c_n)) \subset \pmcg(R'_1)$,  we can consider the  composition
	 	\[ 
	 	\pmcg(S_{c_1\cup c_n})   \xrightarrow{\varphi_1 \circ \sigma}  \mathcal{Z}_0({\varphi_{1}}_*(c_n)) \xrightarrow{\cut_{{\varphi_{1}}_*(c_n)}} \prod_i \pmcg(O'_i),
	 	\]
	 	where the $O'_i$ are the connected components of $R'_1 \setminus {\varphi_{1}}_*(c_n)$. Define  the projection map to each component $\pmcg(O'_i)$ as \[\psi_i = \pi_i \circ \cut_{{\varphi_{1}}_*(c_n)} \circ \varphi_1 \circ \sigma.\] 
	 	It follows from Lemma \ref{lemma:orden_finito_reducible} that $\psi_i(d_{c_1})$ has finite order $k$ for every $i$. 
	 	
	 	The rest of the proof is a bit technical and so we divide it into four steps. In the first step, we show that there exists $m_i\in \mathbb{N}$ such that $\psi_i(d_{c_1}) = \psi_i(d_{c_n})^{m_i}$. In the second step, we show that the element $\varphi(d_{c_1})$ fixes the curves $\varphi_*(c_2)$ outside an open neighborhood of $\varphi_{*}(c_1)$. In the third step, we show that every curve of $\varphi_{*}(c_1)$ intersects at least one curve of $\varphi_{*}(c_2)$ (and vice versa).  Finally, in the four step we show that $\varphi(d_{c_i})=M_i h_i$.
	 	
	 	\textbf{Step 1.}
	 	First, we show $\psi_i(d_{c_n})$ has finite order $k$ for every $i$. 
	 	
	 	Let $A=\varphi_{*}(c_1) \cup \varphi_{*}(c_n)$. Note  $\varphi(\mathcal{Z}_0(c_1)\cap \mathcal{Z}_0(c_n) )$ is contained in  $\mathcal{Z}_0(A)$.  Therefore, we may consider the composition
	 	\[ \mathcal{Z}_0(c_1)\cap \mathcal{Z}_0(c_n) \xrightarrow{\varphi} \mathcal{Z}_0(A) \xrightarrow{\cut_A} \prod_i \pmcg(W'_i),
	 	\]
	 	where the $W'_i$ are the connected  components of  $S' \setminus A$.  To simplify notation, for each projection $\pi_i$ to the ith coordinate we write 
	 	\[ \varphi(f)|_{W'_i} = \pi_i \circ \cut_A \circ \varphi (f) \;\;\text{for any }f \in  \mathcal{Z}_0(c_1)\cap \mathcal{Z}_0(c_n).\]
	 	
	 	Recall that 
	 	\[S' \setminus \varphi_{*}(c_1) = \sqcup_i R'_i\text{  and     }R'_1\setminus {\varphi_{1}}_*(c_n) =  \sqcup_i O'_i.\]
	 	And observe that the surfaces $\{O'_i\}$ are a subset of the surfaces   $\{W'_i\}$. Since $\varphi_i$ is non-multitwist preserving only for $i=1$, if $W'_i \not \in \{O'_j\}$  then the element $\varphi(d_{c_n})|_{W'_i}$ is trivial. Additionally, by Lemma \ref{lemma:orden_finito_reducible}, $\psi_i(d_{c_1})$ has order $k$ for every $O'_i$. These two facts imply
	 	\[ 
	 	\{W'_i|\: \varphi(d_{c_n})|_{W'_i} \ne \text{id}\}=\{O'_i|\: \psi_i(d_{c_n}) \ne \text{id}\} \subset \{O'_i\} \subset \{W'_i|\: \varphi(d_{c_1})|_{W'_i} \ne \text{id}\}.
	 	\]
	 	In words, after cutting by $A$, the set of subsurfaces for which $\varphi(d_{c_n})|_{W'_i}$ is nontrivial is a subset of the subsurfaces for which $\varphi(d_{c_1})|_{W'_i}$ is nontrivial.  However, the argument is symmetric, we may cut first along the multicurve  $\varphi_{*}(c_n)$ and then along the multicurve  $\varphi_{*}(c_1)$. This implies that, after cutting by $A$, the set of subsurfaces for which $\varphi(d_{c_n})|_{W'_i}$ is nontrivial is precisely the set of subsurfaces for $\varphi(d_{c_1})|_{W'_i}$ is nontrivial, that is,  both sets  coincide with  
	 	\[ \{O'_i|\: O'_i\text{ is a connected component of }R'_1\setminus {\varphi_{1}}_*(c_n) \}.\]
	 	
	 	Tracing back the definitions of the maps, we have that $\varphi(d_{c_1})|_{O'_i} = \psi_i(d_{c_1})$   and  $\varphi(d_{c_n})|_{O'_i} = \psi_i(d_{c_n})$.  Observe $\psi_i(d_{c_1}), \psi_i(d_{c_n})$ are commuting finite order elements \new{of order $k$} and together generate a finite abelian subgroup of  $\pmcg(O'_i)$.  Since $O'_i$ is a punctured surface, all  finite abelian subgroups of $\pmcg(O'_i)$ are cyclic. It follows that there exists $m_i\in \mathbb{N}$ such that
	 	\begin{equation}\label{eq:una_potencia_de_la_otra}
	 		\psi_i(d_{c_1}) = \psi_i(d_{c_n})^{m_i}.
	 	\end{equation}

	 	\textbf{Step 2.}
	 	After cutting $S'$ by the multicurve $\varphi_{*}(c_1)$, each curve in $\varphi_{*}(c_2)$ projects onto $R_i'$ either to a curve or to an arc joining punctures. The set containing such   curves and arcs is denoted ${\varphi_{i}}_*(c_2)$.  We start by showing that $\varphi_i(d_{c_1})\cdot a = a$ for every curve or arc $a\in {\varphi_{i}}_*(c_2)$ and  that $\varphi_i(d_{c_1})$ also fixes  the orientation of $a$.
	 	
	 	For $i\ne 1$, the element   $\varphi_i(d_{c_1})$ is the identity. So, we trivially obtain $\varphi_i(d_{c_1}) \cdot a = a$  for every $a\in {\varphi_{i}}_*(c_2)$.
	 	
	 	Assume $i=1$. The curves $\varphi_{*}(c_n)$ are disjoint from the curves in $\varphi_{*}(c_2)$. Therefore, the curves ${\varphi_{1}}_*(c_n)$ are disjoint from the curves/arcs in ${\varphi_{1}}_*(c_2)$. Denote by ${\psi_{i}}_*(c_2)$ the set of curves/arcs of ${\varphi_{1}}_*(c_2)$ that are entirely contained in $O'_i$.  To show that  $\varphi_1(d_{c_1})$ fixes  $a\in {\varphi_{1}}_*(c_2)$ with orientation, it is enough to see that  $\psi_i(d_{c_1})$ fixes  each curve/arc in  ${\psi_{i}}_*(c_2)$ with orientation.
	 	
	 	By Step 1, we have that  $\psi_i(d_{c_1}) = \psi_i(d_{c_n})^{m_i}$ for certain  $m_i\in \mathbb{N}$. Also, by Lemma \ref{lema:si_fijas_curvas_con_orientacion_fijas_imagen_de_esa_curva_con_orient},  $\varphi(d_{c_n})$ fixes every curve in  $\varphi_*(c_2)$ with orientation. We conclude that $\psi_i(d_{c_n})$ fixes each curve/arc in  ${\psi_{i}}_*(c_2)$ with orientation. Now, from Step 1 we deduce that $\psi_i(d_{c_1})$ also fixes each curve/arc in ${\psi_{i}}_*(c_2)$, as we wanted.
	 	
	 	Since $\varphi_i(d_{c_1})$ fixes every curve and arc in ${\varphi_{i}}_*(c_2)$ with orientation, it follows that $\varphi(d_{c_1})$ fixes the curves $\varphi_{*}(c_2)$ outside an open regular neighborhood of $\varphi_{*}(c_1)$.
	 	
	 	\textbf{Step 3.}
	 	Every curve  $x\in \varphi_{*}(c_2)$ intersects a curve in $\varphi_{*}(c_1)$ (and vice versa). Indeed, from Step 2 it follows that if $x\in \varphi_{*}(c_2)$ is disjoint from every curve in $\varphi_{*}(c_1)$, then $\varphi(d_{c_1})\cdot x = x$. Completing $\{d_{c_1}, d_{c_2}\}$ to a Humphries generating set \new{ in a way that new curves are disjoint from $c_2$} and using  Lemma \ref{lema:si_fijas_curvas_con_orientacion_fijas_imagen_de_esa_curva_con_orient}, we deduce that $\varphi$ fixes the curve $x$, which contradicts that $\varphi$ fixes no curve.

	 	\textbf{Step 4.}
	 	It follows from Step 2 that there exists  a multitwist $M_1$ with components in the multicurve  $\varphi_{*}(c_1)$ such that  $\varphi(d_{c_1}) M_1^{-1} \in \pmcg(S')$ fixes every curve in  $\varphi_{*}(c_2)$. Denote $h_1 = \varphi(d_{c_1}) M_1^{-1}$.

	 	Choose a Humphries generating set $\mathcal{H}$ such that $d_{c_1}, d_{c_2}, d_{c_n} \in \mathcal{H}$ and such that $i(c_1, c)=0$ for every $d_{c}\in \mathcal{H}\setminus\{d_{c_2}\}$ (see Figure \ref{fig:curves}). Let $H=\{c|\: d_c\in \mathcal{H}\}$. By definition, $h_1$ fixes every curve  $x\in \varphi_{*}(c_2)$. And, since $d_{c_1}$ commutes with every Dehn twist $d_{c}\in \mathcal{H}\setminus\{d_{c_2} \}$, \new{using Lemma \ref{lema:si_fijas_curvas_con_orientacion_fijas_imagen_de_esa_curva_con_orient}} we have that $h_1$ fixes every curve $x\in \varphi_{*}(c)$ for every $c\in H$.   
	 	
	 	We show $h_1$ has finite order. Let $Y$ be the surface filled by the curves in $\bigcup_{c_i\in H}\varphi_{*}(c_i)$ \new{and note that no three distinct curves in $\bigcup_{c_i\in H}\varphi_{*}(c_i)$ pairwise intersect}. If $Y=S'$, then the Alexander  method yields that $h_1$ has finite order. 
	 	
	 	Looking for a contradiction, assume $Y\ne S'$ and $h_1$ has infinite order. In this case, the Alexander method implies there is an $m\in \mathbb{N}$ such that  $h_1^m$ is supported on the complement of $Y$. Note that $h_1^k = \varphi(d_{c_1}^k) M_1^{-k}$ is a multitwist with components in $\varphi_{*}(c_1)$. Thus, $h_1^{km}$ is  a multitwist supported on both $\varphi_{*}(c_1)$ and the complement of $Y$. That is,  the components of $h_1^{km}$ are contained in both the multicurve $\varphi_{*}(c_1)$ and in the boundary $\partial Y$. However, by Step 3, no curve is both in    $\varphi_{*}(c_1)$ and   $\partial Y$.  Thus, $h_1$ has finite order.

	 	We write $\varphi(d_{c_1}) = M_1 h_1$ \new{and observe that  $h_1^k=\text{id}$, since $h_1^k=\varphi(d_{c_1})^kM_1^{-k}$ is a finite order multitwist.  }
	 	
	 	Finally, the elements $h_1$ and $M_1$  do not depend on the choice of curve  $c_2$ and the choice of generating set $\mathcal{H}$. To see this, consider another curve $c'_2$ and a different generating set $\mathcal{H}'$, we  obtain  $h_1'$ and $M'_1$ with $M_1h_1 = M'_1h'_1$. By taking $k$-th powers on both sides of the equation, we conclude that  $M_1^k=(M'_1)^k$. Since $M_1,M'_1$ are multitwists, it follows that  $M_1=M'_1$. Then,  $h_1=h'_1$.
	 	
	 	For the  curves $c_2, c_3, \dots$ in $H$ we may similarly define  an element of finite order $h_i$ and a multitwist $M_i$ such that
	 	\[ \varphi(d_{c_i}) = M_i h_i.\]
	 	
    \end{proof}
    
    	 \begin{proof}[Proof of Claim \ref{propiedad:los_elementos_se_trenzan}]
    	Since $d_{c_i}^{d_{c_i}d_{c_j}}=d_{c_j}$,  by taking $k$-th powers and taking the image by $\varphi$ we obtain
    	\[ \left(M_i^{\varphi(d_{c_i}d_{c_j})} \right)^k = M_j^k. \]
    	Both $M_j$ and $M_i^{\varphi(d_{c_i}d_{c_j})}$ are multitwists, and so   $M_j = M_i^{\varphi(d_{c_i}d_{c_j})}$. Again, using $d_{c_i}^{d_{c_i}d_{c_j}}=d_{c_j}$ and taking the image by  $\varphi$, we have
    	\[ 
    	M_i^{\varphi(d_{c_i}d_{c_j})} h_i ^{\varphi(d_{c_i}d_{c_j})} = M_j h_j.
    	\]
    	It follows that   $h_i^{\varphi(d_{c_i}d_{c_j})}=h_j$.

    	Lastly, observe that  $h_m$ fixes the curves in  $\varphi_{*}(c_n)$ for every $m, n$. Thus, $h_m$ commutes with  $M_n$ for every $m, n$. Now, from the  equality $h_i^{\varphi(d_{c_i}d_{c_j})}=h_j$  we obtain  
    	\[ (h_ih_j) h_i (h_ih_j)^{-1} = h_j,\]
    	and we conclude that  $h_ih_jh_i=h_jh_ih_j$. Similarly, from $M_j = M_i^{\varphi(d_{c_i}d_{c_j})}$ one deduces $M_iM_jM_i=M_jM_iM_j$.
    \end{proof}
	
	The proofs of Claim \ref{claim:descomposicion} and Claim \ref{propiedad:los_elementos_se_trenzan} complete the argument for Theorem \ref{thm:clasificacion}.
	\end{proof}

\section{Lifting to the diffeomorphism group}\label{sec:lift}
In this section we  consider  surfaces to be endowed with a differentiable structure.  Recall that $\diffeo_c^+(S)$ is the group of orientation preserving diffeomorphisms of $S$ with compact support on $S\setminus \partial S$. In particular, the elements of $\diffeo_c^+(S)$ fix the boundary of $S$ pointwise and a neighbourhood of each puncture. 

Every element in the pure mapping class group $[\phi] \in \pmcg(S)$ admits a representative in $\diffeo_c^+(S)$. Indeed, this follows from the Dehn-Lickorish theorem and the fact that every Dehn twist admits such representative. As a consequence, the map $\pi:\diffeo_c^+(S) \to \pmcg(S) $ sending each diffeomorphism to its homotopy class is a surjective homomorphism. 

In this section we show that every homomorphism $\varphi:\pmcg(S)\to \pmcg(S')$ induced by a multi-embedding can be lifted to a map between the diffeomorphism groups. In other words, there exists $\overline{\varphi}$ such that the following diagram commutes

 \begin{center}
	\begin{tikzcd}
		\diffeo^+_c(S) \arrow[r, "\overline{\varphi}"] \arrow[d, "\pi"] & \diffeo^+_c(S') \arrow[d, "\pi"] \\
		\pmcg(S) \arrow[r, "\varphi"]                          & \pmcg(S').       
	\end{tikzcd}
\end{center}

\begin{proposition}\label{prop:lifts}
	Let $S, S'$ be finite-type orientable smooth surfaces and let $g\geq 3$ be the genus of $S$. If $\varphi:\pmcg(S)\to \pmcg(S')$ is a homomorphism induced by a multi-embedding, then $\varphi$ admits a lift $\overline{\varphi}: \diffeo^+_c(S)\to \diffeo^+_c(S')$.
\end{proposition}
\begin{proof}
	Let $\mathcal{I}=\{\iota_i: S\xhookrightarrow{} S'|\;\;i=1,2,\dots, n\}$ be a multi-embedding inducing the map $\varphi$. We require that the  images of the embeddings have disjoint  closures, i.e, $\overline{\iota_i(S)} \cap \overline{\iota_j(S)}=\emptyset$ for every $i\ne j$.

	\begin{claim}\label{claim_lift}
		Each embedding $\iota_i:S\to S'$  is isotopic to a  map $\epsilon_i: S\to \iota_i(S)\subset  S'$, where $\epsilon_i$ is smooth in the interior of $S$ and is a diffeomorphism onto its image. 
	\end{claim}
	\begin{proof}[Proof of Claim  \ref{claim_lift}]
		The embedding $\iota_i: S\to S'$ is a homeomorphism onto its image. By \cite[Theorem B]{hatcher_kirby_2022} there exists an isotopy $H: S\times [0,1] \to \iota_i(S)$ such that $H(\cdot, 0)=\iota_i(\cdot)$, $H(\cdot, t)=h_t(\cdot)$ is a homeomorphism for every $t$ and $H(\cdot, 1)=\epsilon_i(\cdot)$ is a diffeomorphism. 
	\end{proof}
	
	Now, we replace  $\mathcal{I}=\{\iota_i: S\xhookrightarrow{} S'|\;\;i=1,2,\dots, n\}$ with the multi-embedding $\mathcal{E}=\{\epsilon_i: S\xhookrightarrow{} S'|\;\;i=1,2,\dots, n\}$. Since we only changed the embeddings by isotopies, $\mathcal{E}$ induces the same homomorphism $\varphi$. Moreover, for any $\phi\in \diffeo_c^+(S)$ we can define 
	 \[ 
	\overline{\phi}(x) = 
	\begin{cases}
		x & x\not \in \epsilon_i(S) \; \forall i,\\ 
		\epsilon_i \circ h \circ \epsilon_i^{-1}(x) & x\in \epsilon_i(S).
	\end{cases}
	\] 
	The map $\overline{\phi}$ is a compactly supported diffeomorphism. It follows that the assignation $\phi \mapsto \overline{\phi}$ defines a homomorphism \[ \overline{\varphi}: \diffeo_c^+(S)\to \diffeo_c^+(S') \]
	and the map $\overline{\varphi}$ is a lift of $\varphi$.
\end{proof}

\begin{proof}[Proof of Theorem \ref{thm:version_souto}]
	The statement follows immediately from  Theorem \ref{thm:clasificacion} and Proposition \ref{prop:lifts}.
\end{proof}

\appendix

\section{Homomorphisms not induced by multi-embeddings}\label{appendix:prueba_lema}

Denote  $S_g$ the closed orientable surface of genus $g\in \mathbb{N}$  and denote   $S_{g,1}$  the surface $S_g$ minus a point. 
\begin{theorem}\label{lema:cota_especifica}
	For every $g\geq 2$ and $g' = (g-1)\cdot 2^{2g}+1$, there exists  an injective homomorphism $\pmcg(S_{g,1})\to \pmcg(S_{g', 1})$ not induced by multi-embeddings. 
\end{theorem}
\begin{proof}
	Every characteristic cover  $p: S_{g'}\to S_g$ induces an injective homomorphism  $\varphi: \pmcg(S_{g,1})\to \pmcg(S_{g',1})$ (see \cite[Section 2]{ivanov_injective_1999}). To construct $\varphi$ we choose an appropriate cover. Then, we show that $\varphi$ is not induced by a multi-embedding.
	
	For the cover, consider the abelianzation of  $\pi_1(S_{g})$  composed with the map that quotients each factor modulo two
	\[   \pi_1(S_g)  \xrightarrow{ab}  \mathbb{Z}^{2g} \xrightarrow{\text{mod}2} (\mathbb{Z}/2\mathbb{Z} )^{2g}. \]
	The kernel of the composition produces a characteristic cover  $p: \widetilde{S} \to S_{g}$. The Riemann-Hurwitz formula implies $\widetilde{S}$  has genus  $g'= (g-1)\cdot 2^{2g}+1$. In other words, we may write the cover as  $p:S_{g'}\to S_g$. Let $\varphi:\pmcg(S_{g,1})\to \pmcg(S_{g',1})$ be the injective map induced by $p$. We are left to show that such $\varphi$ is not induced by a multi-embedding. 
	
	It is a consequence of Lemma \ref{lema:colec_emb_induc} that if \new{$Y$ is an orientable finite-type surface and }  $\pmcg(S_{g, 1})\to \pmcg(Y)$ a multi-embedding induced homomorphism, then   $Y$ is homeomorphic to either $S_g$, $S_{g, 1}$ or $S_{2g}$. Since $S_{g',1}$ has genus $g'>2g$, we conclude $\varphi$ is not induced by a multi-embedding.   
\end{proof}


\printbibliography

\end{document}